\documentclass[english,reqno]{smfart}
\usepackage{etex}
\usepackage [utf8]{inputenc}
\usepackage[T1]{fontenc}
\usepackage{amssymb}
\usepackage{amsrefs}
\usepackage{smfthm}
\usepackage{bull}
\usepackage{pstricks}
\usepackage{pst-plot}
\usepackage{eso-pic}
\usepackage{graphicx}
\usepackage{tikz}
\usepackage{listings}
\usepackage[cmtip,all]{xy}

\usepackage{fancybox}

\newtheorem{theoa}{Theorem}
\newtheorem{propB}{Proposition}

\renewcommand{\thepropB}{}

\newcommand{\PP}{\mathbb{P}}

\newcommand{\ZZ}{\mathbb{Z}}

\newcommand{\CC}{\mathbb{C}}

\newcommand{\HH}{{\mathbb{H}}}

\newcommand{\QQ}{\mathbb{Q}}
\newcommand{\Fb}{\mathbf{F}}
\newcommand{\Cr}{\mathcal{C}}
\newcommand{\Qr}{\mathcal{Q}}
\newcommand{\Or}{\mathcal{O}}

\newcommand{\Er}{\mathcal{E}}

\newcommand{\Fr}{\mathcal{F}}

\newcommand{\dr}{\partial}
\newcommand{\Ric}{\mathfrak{R}}
\newcommand{\dd}{\mathrm{d}}
\newcommand*{\DashedArrow}[1][]{\mathbin{\tikz [baseline=-0.25ex,-latex, dashed,#1] \draw [#1] (0pt,0.5ex) -- (1.3em,0.5ex);}}

\title[Isomonodromic deformations over the five punctured sphere]{A new two--parameter family of isomonodromic deformations over the five punctured sphere}

\author{Arnaud Girand}
\address{IRMAR (UMR 6625 du CNRS), Universit\'e de Rennes 1, campus de Beaulieu. B\^at. 22-23. 35042 Rennes Cedex}
\curraddr{}
\email{arnaud.girand@univ-rennes1.fr}
\thanks{}

\subjclass{14E22, 20G05, 20G20, 32D15, 32G08, 32G34, 34M50, 34M56, 51N15, 55R10}

\begin{document}

\frontmatter

\date{}

\dedicatory{}

\begin{abstract}
The object of this paper is to describe an explicit two--parameter family of logarithmic flat connections over the complex projective plane. These connections have dihedral monodromy and their polar locus is a prescribed quintic composed of a conic and three tangent lines. By restricting them to generic lines we get an algebraic family of isomonodromic deformations of the five--punctured sphere. This yields new algebraic solutions of a Garnier system. Finally, we use the associated Riccati one--forms to construct and prove the integrability (in the transversally projective sense) of a subfamily of Lotka--Volterra foliations. 
\end{abstract}

\begin{altabstract}
Le but de cet article est de décrire une famille explicite à deux paramètres de connexions logarithmiques plates au dessus du plan projectif complexe. Ces connexions sont à monodromie diédrale et leur lieu polaire est une quintique prescrite, composée d'une conique et de trois droites tangentes. Par restriction aux droites génériques, on obtient alors une famille algébrique de déformations isomonodromiques de la sphère à cinq trous. Ceci livre de nouvelles solutions algébriques d'un système de Garnier. Enfin, nous utilisons les formes de Riccati associées à ces connexions pour construire et montrer l'intégrabilité (au sens transversalement projectif) d'une sous-famille de feuilletages de Lotka--Volterra.
\end{altabstract}

\maketitle
\noindent \textbf{Last modified: \today.}
\setcounter{tocdepth}{1}

\noindent\shadowbox{\begin{minipage}{\textwidth}\noindent \textbf{Disclaimer :} \textit{this is an author--created, un--copyedited version of an article accepted for publication in} Bulletin de la Société Mathématique de France\textit{. The Version of Record is available online at}  
\footnotesize\verb+http://smf4.emath.fr/Publications/Bulletin/144/html/smf_bull_144_339-368.php+ \textit{.}\end{minipage}}

\normalsize

\mainmatter

\section{Introduction} 

In this section we describe the main result of this paper, namely an explicit construction of a two--parameter family of logarithmic flat connections over the complement of a particular quintic curve in $\PP^2_\CC$. The restriction of any element of this family to generic lines in the projective plane gives a isomonodromic deformation over the five punctured sphere, to which we can associate an algebraic solution of some Hamiltonian system of partial differential equations, namely the Garnier-$2$ system.

\subsection{Topology of the complement of a particular plane quintic}\label{sec:topology} In this paper, we concern ourselves with setting up a two--parameter family of logarithmic flat $\mathfrak{sl}_2(\CC)$--connections over $\PP^2$ with a specific polar locus, namely a quintic curve $\Qr$ composed of a circle and three tangent lines. More precisely, in homogeneous coordinates $[x:y:t]$, $\Qr$ is defined, up to $PGL_3(\CC)$ action, by the equation
\begin{displaymath}
xyt(x^2 + y^2 + t^2 - 2(xy+xt+yt)) = 0 \; . 
\end{displaymath}

Before stating our main result, let us specify what we are looking for: we want to find a family of rank two logarithmic flat connections over $\PP^2$ with polar locus equal to some small degree curve and "interesting monodromy". We will show that it is possible to do so with the quintic $\Qr$ define above.

\begin{defi}\label{def:nonDeg}
We say that the monodromy representation associated with a rank two logarithmic flat $\mathfrak{sl}_2(\CC)$--connection over $\PP^2-\Qr$ is non--degenerate if
\begin{itemize}
\item its image forms an irreducible subgroup of $SL_2(\CC)$ ; 
\item its local monodromy (see Definition \ref{def:locMon}) around any irreducible component of $\Qr$ is projectively non--trivial (i.e is non-trivial in $PSL_2(\CC)$).
\end{itemize}
\end{defi}

In order to establish the existence of such representations, we use the following result by Degtyarev.

\begin{prop}\label{th:Degtyarev}{ \rm \textbf{[Degtyarev, 1999 \cite{Degtyarev}] }}
The fundamental group $\Gamma$ of the complement of a smooth conic and three tangent lines in $\PP^2$ admits the following presentation:
\begin{displaymath}
\Gamma \cong \langle a,b,c \, | \, (ab)^2 (ba)^{-2} = (ac)^2 (ca)^{-2}= [b,c]=1 \rangle \;  .
\end{displaymath}
\end{prop}

\begin{center}
\begin{figure}

\scalebox{0.6} 
{
\begin{pspicture}(0,-3.85)(17.841875,3.85)
\definecolor{color3}{rgb}{0.00392156862745098,0.00392156862745098,0.00392156862745098}
\definecolor{color18}{rgb}{0.0196078431372549,0.611764705882353,0.00392156862745098}
\definecolor{color36}{rgb}{0.0392156862745098,0.9529411764705882,0.09803921568627451}
\definecolor{color20}{rgb}{0.19215686274509805,0.6274509803921569,0.00392156862745098}
\pscircle[linewidth=0.04,linecolor=red,dimen=outer](6.53,-0.28){1.79}
\psline[linewidth=0.04cm,linecolor=color3](0.16,-3.83)(7.66,3.23)
\psline[linewidth=0.04cm,linecolor=color3](4.32,3.83)(11.12,-1.71)
\psline[linewidth=0.04cm,linecolor=color3](0.0,-3.47)(13.8,-0.57)
\psline[linewidth=0.04cm,linecolor=color18](14.3,-0.65)(1.44,-0.85)
\usefont{T1}{ptm}{m}{n}
\rput(4.7042184,-1.9){\color{red}$\Cr$}
\usefont{T1}{ptm}{m}{n}
\rput(11.3142185,-0.445){\color{color36}$L$}
\pscustom[linewidth=0.04]
{
\newpath
\moveto(9.02,-0.73)
\lineto(8.96,-0.83)
\curveto(8.93,-0.88)(8.88,-0.975)(8.86,-1.02)
\curveto(8.84,-1.065)(8.795,-1.17)(8.77,-1.23)
\curveto(8.745,-1.29)(8.705,-1.39)(8.69,-1.43)
\curveto(8.675,-1.47)(8.65,-1.56)(8.64,-1.61)
\curveto(8.63,-1.66)(8.62,-1.75)(8.62,-1.79)
\curveto(8.62,-1.83)(8.62,-1.915)(8.62,-1.96)
\curveto(8.62,-2.005)(8.635,-2.085)(8.65,-2.12)
\curveto(8.665,-2.155)(8.71,-2.2)(8.74,-2.21)
\curveto(8.77,-2.22)(8.825,-2.23)(8.85,-2.23)
\curveto(8.875,-2.23)(8.92,-2.225)(8.94,-2.22)
\curveto(8.96,-2.215)(9.0,-2.19)(9.02,-2.17)
\curveto(9.04,-2.15)(9.075,-2.105)(9.09,-2.08)
\curveto(9.105,-2.055)(9.13,-1.995)(9.14,-1.96)
\curveto(9.15,-1.925)(9.16,-1.86)(9.16,-1.83)
\curveto(9.16,-1.8)(9.155,-1.75)(9.15,-1.73)
\curveto(9.145,-1.71)(9.12,-1.66)(9.1,-1.63)
}
\pscustom[linewidth=0.04]
{
\newpath
\moveto(9.02,-1.51)
\lineto(8.95,-1.47)
\curveto(8.915,-1.45)(8.86,-1.425)(8.84,-1.42)
\curveto(8.82,-1.415)(8.785,-1.4)(8.77,-1.39)
}
\pscustom[linewidth=0.04]
{
\newpath
\moveto(9.02,-0.71)
\lineto(9.04,-0.63)
\curveto(9.05,-0.59)(9.06,-0.52)(9.06,-0.49)
\curveto(9.06,-0.46)(9.065,-0.395)(9.07,-0.36)
\curveto(9.075,-0.325)(9.085,-0.27)(9.09,-0.25)
\curveto(9.095,-0.23)(9.105,-0.19)(9.11,-0.17)
\curveto(9.115,-0.15)(9.115,-0.105)(9.11,-0.08)
\curveto(9.105,-0.055)(9.105,0.0)(9.11,0.03)
\curveto(9.115,0.06)(9.125,0.11)(9.13,0.13)
\curveto(9.135,0.15)(9.145,0.19)(9.15,0.21)
\curveto(9.155,0.23)(9.18,0.26)(9.2,0.27)
\curveto(9.22,0.28)(9.265,0.295)(9.29,0.3)
\curveto(9.315,0.305)(9.355,0.31)(9.37,0.31)
\curveto(9.385,0.31)(9.425,0.305)(9.45,0.3)
\curveto(9.475,0.295)(9.52,0.27)(9.54,0.25)
\curveto(9.56,0.23)(9.6,0.195)(9.62,0.18)
\curveto(9.64,0.165)(9.675,0.13)(9.69,0.11)
\curveto(9.705,0.09)(9.715,0.045)(9.71,0.02)
\curveto(9.705,-0.005)(9.695,-0.06)(9.69,-0.09)
\curveto(9.685,-0.12)(9.665,-0.17)(9.65,-0.19)
\curveto(9.635,-0.21)(9.605,-0.24)(9.59,-0.25)
\curveto(9.575,-0.26)(9.545,-0.28)(9.53,-0.29)
}
\pscustom[linewidth=0.04]
{
\newpath
\moveto(9.36,-0.37)
\lineto(9.29,-0.38)
\curveto(9.255,-0.385)(9.195,-0.395)(9.17,-0.4)
\curveto(9.145,-0.405)(9.09,-0.42)(9.06,-0.43)
}
\pscustom[linewidth=0.04]
{
\newpath
\moveto(9.02,-0.71)
\lineto(8.95,-0.67)
\curveto(8.915,-0.65)(8.86,-0.61)(8.84,-0.59)
\curveto(8.82,-0.57)(8.78,-0.53)(8.76,-0.51)
\curveto(8.74,-0.49)(8.695,-0.455)(8.67,-0.44)
\curveto(8.645,-0.425)(8.59,-0.39)(8.56,-0.37)
\curveto(8.53,-0.35)(8.485,-0.325)(8.47,-0.32)
\curveto(8.455,-0.315)(8.425,-0.3)(8.41,-0.29)
\curveto(8.395,-0.28)(8.355,-0.27)(8.33,-0.27)
\curveto(8.305,-0.27)(8.255,-0.27)(8.23,-0.27)
\curveto(8.205,-0.27)(8.16,-0.28)(8.14,-0.29)
\curveto(8.12,-0.3)(8.09,-0.325)(8.08,-0.34)
\curveto(8.07,-0.355)(8.05,-0.39)(8.04,-0.41)
\curveto(8.03,-0.43)(8.035,-0.465)(8.05,-0.48)
\curveto(8.065,-0.495)(8.105,-0.525)(8.13,-0.54)
\curveto(8.155,-0.555)(8.22,-0.58)(8.26,-0.59)
}
\pscustom[linewidth=0.04]
{
\newpath
\moveto(8.32,-0.61)
\lineto(8.37,-0.61)
\curveto(8.395,-0.61)(8.445,-0.605)(8.47,-0.6)
\curveto(8.495,-0.595)(8.535,-0.58)(8.55,-0.57)
\curveto(8.565,-0.56)(8.595,-0.54)(8.61,-0.53)
\curveto(8.625,-0.52)(8.67,-0.5)(8.7,-0.49)
}
\psdots[dotsize=0.12,linecolor=color3,fillstyle=solid,dotstyle=o](9.9,-0.71)
\psdots[dotsize=0.12,linecolor=color3,fillstyle=solid,dotstyle=o](8.22,-0.73)
\psdots[dotsize=0.12,linecolor=color3,fillstyle=solid,dotstyle=o](4.84,-0.79)
\psdots[dotsize=0.12,linecolor=color3,fillstyle=solid,dotstyle=o](3.36,-0.83)
\usefont{T1}{ptm}{m}{n}
\rput(8.58422,-0.125){$a$}
\usefont{T1}{ptm}{m}{n}
\rput(9.394218,-1.765){$b$}
\usefont{T1}{ptm}{m}{n}
\rput(9.944219,0.175){$c$}
\psframe[linewidth=0.04,linecolor=color20,dimen=outer](17.12,3.37)(14.74,-2.77)
\psdots[dotsize=0.12,linecolor=color3,fillstyle=solid,dotstyle=o](13.36,-0.67)
\psbezier[linewidth=0.04](15.98,0.23)(15.98,-0.57)(16.209753,-0.556751)(15.98,-1.53)(15.750247,-2.503249)(16.910694,-1.8503163)(15.94,-1.61)
\psbezier[linewidth=0.04](15.986626,0.19675101)(15.986626,0.78128284)(15.742557,0.7754422)(15.986626,1.4827211)(16.230696,2.19)(15.08,1.57)(16.0,1.47)
\psbezier[linewidth=0.04](16.0,-0.13)(15.930369,0.66556793)(16.84,1.63)(16.291744,2.4861145)(15.743488,3.342229)(14.78,2.03)(16.26,2.53)
\psbezier[linewidth=0.04](15.991281,0.10603877)(15.836095,-0.21)(16.030344,-0.4149751)(15.275172,-0.17836133)(14.52,0.05825244)(15.026629,-1.1368483)(15.213098,-0.13716686)
\psbezier[linewidth=0.04](15.98,0.09)(16.600689,-0.29)(17.18,-0.19)(16.719833,0.34755182)(16.259666,0.88510364)(17.0,1.21)(16.784048,0.30635735)
\psdots[dotsize=0.12,fillstyle=solid,dotstyle=o](15.82,2.59)
\psdots[dotsize=0.12,fillstyle=solid,dotstyle=o](15.86,1.67)
\psdots[dotsize=0.12,fillstyle=solid,dotstyle=o](15.02,-0.31)
\psdots[dotsize=0.12,fillstyle=solid,dotstyle=o](16.72,0.63)
\psdots[dotsize=0.12,fillstyle=solid,dotstyle=o](16.18,-1.85)
\usefont{T1}{ptm}{m}{n}
\rput(14.274219,-2.625){\color{color36}$L$}
\usefont{T1}{ptm}{m}{n}
\rput(15.181406,2.895){$d_5$}
\usefont{T1}{ptm}{m}{n}
\rput(15.292813,1.955){$d_3$}
\usefont{T1}{ptm}{m}{n}
\rput(14.981406,0.2){$d_4$}
\usefont{T1}{ptm}{m}{n}
\rput(15.581407,-1.765){$d_1$}
\usefont{T1}{ptm}{m}{n}
\rput(16.341406,0.735){$d_2$}
\usefont{T1}{ptm}{m}{n}
\rput(10.151406,-0.45){$4$}
\usefont{T1}{ptm}{m}{n}
\rput(13.3,-0.35){$1$}
\usefont{T1}{ptm}{m}{n}
\rput(8.4,-0.91){$2$}
\usefont{T1}{ptm}{m}{n}
\rput(5.1914062,-0.95){$3$}
\usefont{T1}{ptm}{m}{n}
\rput(3.3,-0.585){$5$}
\end{pspicture} 
}

\caption{Fundamental group of the complement of the quintic $\Qr$ in $\PP^2$ and restriction to a generic line.}
\label{fig:Pi_1_gen}
\end{figure}
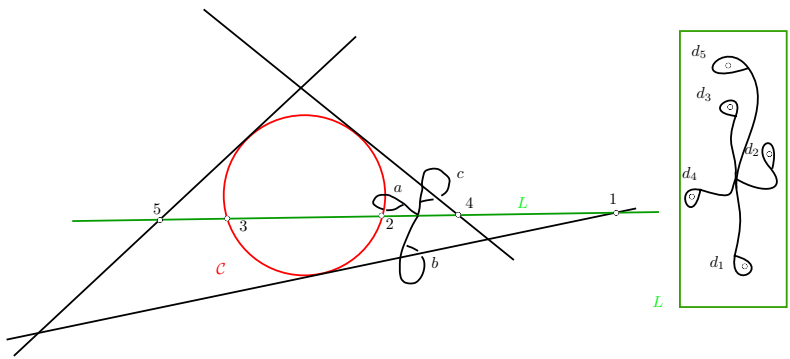

\end{center}

More precisely, Degtyarev proves that we can take $a$ (resp. $b,c$) to be a loop realising the local monodromy (see Definition \ref{def:locMon}) around the conic $\Cr := (x^2 + y^2 + t^2 - 2(xy+xt+yt)=0)$  (resp. the lines $(y=0)$, $(x=0)$), as illustrated in the left--hand side of Fig. \ref{fig:Pi_1_gen}. Also note that the fundamental group of the intersection of $\PP^2-\Qr$ with any generic line is isomorphic to the free group $\Fb_4 := \langle d_1, \ldots, d_5 \, | \, d_1 \ldots d_5 = 1 \rangle$; the Lefschetz hyperplane theorem (see \cite{Milnor}, Theorem 7.4) tells us that the natural morphism $\tau : \Fb_4 \rightarrow \Gamma$ is onto. Moreover we know from the explicit Van Kampen method given in  Subsection 4.1 of \cite{Degtyarev} that the group $\Gamma$ can be computed by taking the four free generators of the fundamental group of the intersection of $\PP^2 - \Qr$ with any generic line and adding some  braid monodromy relations. Thus, if we chose a line going through the base point used to define $a,b$ and $c$ then $\tau$ is given (up to a permutation of the $d_i$) by (see the right--hand side of Fig. \ref{fig:Pi_1_gen}):
\begin{align*}
d_1 & \mapsto  b \\
d_2 & \mapsto a\\
d_3 & \mapsto b a b^{-1}\\
d_4 & \mapsto c\\
d_5 & \mapsto (abac)^{-1} \; .
\end{align*} 
In particular, any non--degenerate representation $\rho$ of $\Gamma$ must satisfy $$\rho(a), \rho(b) , \rho(c), \rho(abac) \neq \pm I_2 \; . $$

\begin{prop}\label{prop:repres}
The only (up to conjugacy) family of non--degenerate representations of $\Gamma$ into $SL_2(\CC)$ is as follows:
$$\rho_{u,v} : a  \mapsto \begin{pmatrix}
0 & 1 \\ 
-1 & 0
\end{pmatrix}  , \quad 
b  \mapsto \begin{pmatrix}
u & 0 \\ 
0 & u^{-1}
\end{pmatrix}, \quad
c  \mapsto \begin{pmatrix}
v & 0 \\ 
0 & v^{-1}
\end{pmatrix} \; , \text{ for }  u,v \in \CC^* \;.$$
\end{prop}

\subsection{Main results}

The core of this paper will be devoted to proving the following theorem, in which we explicitly construct the announced family of rank two logarithmic flat connections.

\begin{theoa}\label{TheoA}
There exists an \emph{explicit} two--parameter family $\nabla_{\lambda_0, \lambda_1}$ of logarithmic flat connections over the trivial rank two vector bundle $\CC^2 \times \PP^2 \rightarrow \PP^2$ with the following properties:
\begin{enumerate}
\item[(i)] the polar locus of $\nabla_{\lambda_0, \lambda_1}$ is equal to the quintic $\Qr\in \PP^2$ and as such does not depend on $\lambda_0, \lambda_1 \in \CC$;
\item[(ii)] the monodromy of $\nabla_{\lambda_0, \lambda_1}$ is conjugated to $\rho_{u,v}$ with $u = -e^{-i \pi \lambda_0}$ and $v = e^{-i \pi \lambda_1}$. It is a virtually abelian dihedral representation of the fundamental group $\Gamma:=\pi_1(\PP^2-\Qr)$ into $SL_2(\CC)$ whose image is not Zariski--dense.
\end{enumerate}
The connection $\nabla_{\lambda_0, \lambda_1}$ is given in some (see Subsection \ref{sec:Rk2VB}) affine chart $\CC^2_{x,y} \subset	\PP^2$ by:
\begin{displaymath}
\nabla_{\lambda_0,\lambda_1} = \dd - \dfrac{1}{2(x^2+y^2+1-2(xy+x+y))} (\lambda_0 A_0 + \lambda_1	A_1 + A_2) \; , 
\end{displaymath}
where\tiny
\begin{displaymath}
A_0 := \begin{pmatrix}
2(x-1)y\dd x + (x^2 + x(y-2) - y +1)x \frac{\dd y}{y} & 2(2x - y +2)y \dd x + (2x^2 +y(x-y+3)-2)x \frac{\dd y}{y} \\ 
-2y^2 \dd x + (x+y-1)x^2\frac{\dd y}{y}  & -2(x-1)y\dd x - (x^2 + x(y-2) - y +1)x\frac{\dd y}{y}
\end{pmatrix} 
\end{displaymath}
\begin{displaymath}
A_1 := \begin{pmatrix}
(x^2+(x-1)(y-1))y \frac{\dd x}{x} + 2(x-1)x \dd y & (x^2 + y(x-y+3) -2)y \frac{\dd x}{x} + 2(2x-y+2)x \dd y \\ 
-(x+y-1)y^2 \frac{\dd x}{x}-2x^2 \dd y & -(x^2+(x-1)(y-1))y \frac{\dd x}{x} - 2(x-1)x \dd y
\end{pmatrix} 
\end{displaymath}
\begin{displaymath}
A_2 := \begin{pmatrix}
-(x+y+1)y \dd x - (x^2-x(y+2)-y+1)x \frac{\dd y}{y} & -2(x-y+3)y \dd x  -(x^2-2y(x+1)+1) x \frac{\dd y}{y}\\ 
0 & (x+y+1)y \dd x + (x^2-x(y+2)-y+1)x \frac{\dd y}{y}
\end{pmatrix} \; .
\end{displaymath}
\normalsize
\end{theoa}

\begin{rema}
Note that the existence, and uniqueness up to gauge transformation, of such a family of connections follows from Proposition \ref{prop:repres} and the classical Riemann--Hilbert correspondence. The original part of this work resides in the fact that we give a constructive proof of this result; in particular this allows us to describe the associated algebraic Garnier solution.
\end{rema}

Since $\PP^2$ is the symmetric product $\mathrm{Sym}^2(\PP^1)$ one has a natural two--fold ramified covering $\pi : \PP^1 \times \PP^1 \xrightarrow[]{2:1} \PP^2$ which pulls the quintic $\Qr$ back onto the subset $D \subset X := \PP^1 \times \PP^1$ composed of the six lines $u_0, u_1 = 0,1,\infty$ (for some pair $(u_0,u_1)$ of projective coordinates on $X$) and of the diagonal $\Delta$ while ramifying over the latter (see Subsection \ref{sec:Rk2VB} for more details). As we are aiming at dihedral monodromy, a natural idea to prove Theorem \ref{TheoA} is to define a family of rank one logarithmic flat connections over $X$ with infinite monodromy around $D\setminus\Delta$ and to push it forward using $\pi$ to get a family of such connections over $\PP^2 - \Qr$ with monodromy of (generically) infinite order around the three lines in the quintic and of projective order two at the conic $\Cr$. This is exactly what we will do in Section \ref{sec:proofTheoA}.

Representations of fundamental groups of quasi-projective varieties in $SL_2(\CC)$ have been classified mainly by Corlette and Simpson \cite{CorSim}. One important class of such representations is that of those factoring through a curve.

\begin{defi}\label{def:factorCurve}  \cite{CorSim, LTP}
We say that a representation $\rho : \Gamma \rightarrow SL_2(\CC)$ \emph{factors through a curve} if there exists a complex projective curve $C$, a divisor $\delta \subset C$, an algebraic mapping $f : \PP^2 -f^{-1}(\delta) \rightarrow C - \delta$ and a representation $\tilde{\rho}$ of the fundamental group of $C-\delta$ into $PSL_2(\CC)$ such that
\begin{enumerate}
\item[(i)]  $f^{-1}(\delta)$ contains $\Qr$, therefore there exists a natural group homomorphism $m : \pi_1(\PP^2 - f^{-1}(\delta)) \rightarrow \Gamma$;
\item[(ii)] the diagram
\begin{displaymath}
\xymatrix{
\pi_1(C-\delta)  \ar[rd]_{\tilde{\rho}} & \ar[l]_{f_*} \pi_1(\PP^2 - f^{-1}(\delta)) \ar[d]^{\mathrm{P}\circ\rho\circ m}\\
 & PSL_2(\CC) 
}
\end{displaymath}
commutes, where $\mathrm{P}$ is the natural morphism $SL_2(\CC) \rightarrow PSL_2(\CC)$.
\end{enumerate}

\end{defi}

Indeed, representations admitting such a factorisation can be obtained through pullback from the monodromy of some logarithmic flat connection on a curve.  We prove that this is not generically the case for our family of connections.

\begin{theoa}\label{TheoB}
The monodromy representation of the connections $\nabla_{\lambda_0,\lambda_1}$ introduced in Theorem \ref{TheoA} factors through a curve if and only if there exists $(p, q) \in \ZZ^2\setminus \lbrace (0,0) \rbrace$ such that $p \lambda_0 + q \lambda_1 = 0$, i.e if and only if $[\lambda_0 : \lambda_1] \in \PP^1_\QQ$. 
\end{theoa} 

\begin{rema}
Note that our choice of this particular quintic is not arbitrary: work in progress \cite{these} using Degtyarev's aforementioned paper suggests that representations of the complement of most quintics in the projective plane are either degenerate or factor through a curve.
\end{rema}

\subsection{Isomonodromic deformations}

By restriction to generic lines in $\PP^2$, we obtain a family (parametrized by $(\lambda_0, \lambda_1)$) of isomonodromic deformations over the Riemann sphere with five pairwise distinct punctures  $\PP^1_x\setminus \lbrace 0,1,t_1,t_2, \infty \rbrace$, whose monodromy is given in Table \ref{table:monodromy}, where $x$ is a well chosen projective coordinate on $\PP^1$ and $t_1,t_2$ are two independent variables (well defined up to double covering) corresponding to the intersection of the line with the conic $\Cr$. Since this family of connections  is algebraic we get a  two--parameter family of algebraic solutions of the isomonodromy equation associated with such deformations, namely the following Garnier system:
\begin{equation}
\left\lbrace \begin{array}{ll}
\dr_{t_k} p_i  & = - \dr_{q_i} H_k \\ 
\dr_{t_k} q_i  & =  \dr_{p_i} H_k \\ 
\end{array} \right. \; i,j =1,2 \quad ,
\end{equation}
where $(p_i,q_i)_i$ are algebraic functions of $t_1,t_2$ and $H_1, H_2$ are explicit Hamiltonians given in Proposition \ref{prop:Garnier} (see also \cites{Mazz, Diar1}). More precisely if one sets $S_q:=q_1+q_2$, $P_q := q_1q_2$, $S_t := t_1+t_2$ and $P_t:=t_1t_2$ one has the following relations:
\begin{equation}
\left\lbrace \begin{array}{l}
(\lambda_0-1)^2 \lambda_1^2 S_t = -F(S_q,P_q) \\ 
(\lambda_0-1)^2 P_t = - (\lambda_0+\lambda_1-1)^2P_q^2
\end{array} \right. \; ;
\end{equation}
where:
\begin{align*}
F(S_q,P_q) = &(\lambda_0-\lambda_1-1)(\lambda_0+\lambda_1-1)^3 P_q^2\\
& + (\lambda_0-1)^2(\lambda_0+\lambda_1-1)^2(2P_q-2P_qS_q+S_q^2-2Sq)\\
&+(\lambda_0-1)^3(\lambda_0+2\lambda_1-1) \; .
\end{align*}
These solutions generalise the two parameter family known for the Painlevé VI equation (see Subsection \ref{sec:PVI}) and the complex surface associated with the graph of $(t_1,t_2) \mapsto (S_q,P_q)$ is rational.

\begin{table}[!!h]
\begin{center}
\begin{tabular}{c|c|c|c|c}
$x=0$ & $x=1$ & $x=t_1$ & $x=t_2$ & $x=\infty$ \\ 
\hline 
& & & & \\
$\begin{pmatrix}
a_1& 0 \\ 
0 & a_1^{-1}
\end{pmatrix}$ & $\begin{pmatrix}
-a_0 & 0 \\ 
0 &  -a_0^{-1}
\end{pmatrix}$ & $\begin{pmatrix}
0 & 1 \\ 
-1 & 0
\end{pmatrix}$ & $\begin{pmatrix}
0 & a_0^{2}\\ 
-a_0^{-2} & 0
\end{pmatrix}$ & $\begin{pmatrix}
a_0a_1^{-1} & 0 \\ 
0 & a_0^{-1} a_1
\end{pmatrix}$ \\ 
& & & & \\
\end{tabular}
\end{center}  
\caption{Monodromy on a generic line; here $a_j = \exp(- i \pi \lambda_j)$ for $j \in \lbrace 0,1 \rbrace$.}
\label{table:monodromy}
\end{table}

\subsection{Lotka--Volterra foliations} 

One last fact worth noting is that since $\nabla$ is a flat $\mathfrak{sl}_2(\CC)$--connection on a trivial bundle, there exist three meromorphic one--forms $\alpha_0, \alpha_1$ and $\alpha_2$ (given in Theorem \ref{TheoA}) such that
\begin{displaymath}
\nabla = \dd + \Omega \text{ , where } \Omega := \begin{pmatrix}
\alpha_1 & \alpha_0 \\ 
-\alpha_2 & -\alpha_1
\end{pmatrix}  \text{ satisfies } \dd \Omega = \Omega \land \Omega \, .
 \end{displaymath} 
In particular, since $\dd \alpha_2 \land \alpha_2 = 0$ one obtains a family of transversally projective degree two foliations over $\PP^2$ (see \cite{LTP} and Section \ref{sec:LV}; note that this family is therefore integrable in the Casale--Malgrange sense \cite{Casale}) with invariant locus containing the quintic $\Qr$. We show that these are conjugate to a family of Lotka-Volterra foliations over $\CC^3$ \cite{MO, MO_cor}; namely given three complex parameters $(A,B,C)$, the codimension one foliation associated with the one--form over $\CC^3$, with coordinates $(x,y,t)$:
 \begin{displaymath}
\omega_0 := (yV_t-tV_y)\dd x + (tV_x - xV_t)\dd y + (xV_y - yV_x) \dd t \; ,
\end{displaymath} 
where:
\begin{displaymath}
V_x := x(Cy+t), \qquad V_y := y(At+x) \qquad \text{ and } \qquad Vt := t(Bx+y) \; .
\end{displaymath}

\begin{theoa}\label{TheoC}
The foliation defined by the meromorphic one--form $\alpha_2$ is equal to the foliation over $\PP^2$ associated with a Lotka--Volterra system with parameters 
 \begin{displaymath}
 (A,B,C) = \left( \dfrac{\lambda_1}{\lambda_0}, \dfrac{-\lambda_0}{\lambda_0+\lambda_1}, \dfrac{-(\lambda_0+\lambda_1)}{\lambda_1}\right) \quad .
 \end{displaymath}
Conversely, any degree two foliation whose invariant locus contains the quintic $\Qr$ is equal to one of the above form.
 \end{theoa}

One can see from Theorem \ref{TheoC} that this family of  foliations is governed by the parameter $\lambda_0/\lambda_1$; there exists a one--parameter family of connections corresponding to any given foliation (see also Subsection 4.4 in \cite{LTP}). We then prove that this gives an example of a family of foliations with algebraic invariant curves of arbitrarily high degree (see also \cite{ALN}).

\subsection*{Acknowledgements} The author would like to thank both Serge Cantat and Frank Loray for their support and guidance in writing this paper. Gaël Cousin, Karamoko Diarra and Valente Ramirez Garcia Luna contributed through several interesting conversations and offered valuable insight on various topics. Special thanks to Thiago Fassarella for the interest he took in this paper and for suggesting several improvements to its exposition. Funding was provided by the Université de Rennes 1, the École normale supérieure de Rennes and the Centre Henri Lebesgue.

Finally, the author thanks the anonymous referee for the time and attention they devoted to reviewing this work.

\section{Proof of Theorem \ref{TheoA}}\label{sec:proofTheoA}

In this section we concern ourselves with setting up a particular fibre bundle over the projective plane $\PP^2$, and then endowing it with a family of logarithmic flat connections satisfying the conditions of Theorem \ref{TheoA}.

\subsection{A rank two fibre bundle}\label{sec:Rk2VB}

Start by considering the complex manifold $X:=\PP^1 \times \PP^1$ and define the following involution: 
\begin{align*}
\tilde{\eta} : X &  \rightarrow X \\
(u_0,u_1) & \mapsto (u_1,u_0) \; .
\end{align*}

The action of $\tilde{\eta}$ gives us a two--fold ramified covering of $\PP^1 \times \PP^1$ over the projective plane $\PP^2$, i.e the fibres of the morphism 
\begin{align*}
 \PP^1 \times \PP^1 & \longrightarrow \PP^2 \\
([u_0^0:u_0^1],[u_1^0:u_1^1]) & \mapsto [s:p:t] := [(u_0^0u_1^1 + u_0^1u_1^0): u_0^0u_1^0:u_0^1 u_1^1 ] \, .
\end{align*}
are the orbits under $\bar{\eta}$. However, for the purpose of this paper, we will compose this mapping with the linear projective transformation of $\PP^2$ given by:
\begin{displaymath}
[s:p:t] \mapsto [p+t-s:p:t] \; .
\end{displaymath}
This means that we will now work in the homogeneous coordinates $x := p+t-s$, $y:=p$ and $t$.
We get a two--fold ramified covering $\pi : X \rightarrow \PP^2$ that ramifies along the diagonal $\Delta :=(u_0=u_1) \subset X$ and sends it onto the conic:
\begin{displaymath}
(\Cr) \qquad x^2+y^2+t^2 = 2(xy +xt+yt) \;.
\end{displaymath}

Now consider the rank two fibre bundle $E$ over the projective plane associated with the locally free sheaf 
\begin{displaymath}
\Er := \Or_{\PP^2} \oplus \Or_{\PP^2} (-1) \, .
\end{displaymath}
Let $e_+$ be some global nonvanishing holomorphic section of $E$ (corresponding to the $\Or_{\PP^2}$ part of the above decomposition) and $e_-$ be some global meromorphic section linearly independent from $e_+$ (and so corresponding to $\Or_{\PP^2} (-1)$) with associated (zeroes and poles) divisor equal to $- L_\infty$, where $L_\infty$ is the line "at infinity" $(t=0)$.

Let us ask ourselves the following question: what does the pullback sheaf $\Fr := \pi^* \Er$ look like ? For any open set $U \subset X$ we have $\Fr(U) = \Er( \pi(U))$, which implies that $\Fr$ is a rank two locally free sheaf inducing a rank two fibre bundle $F \rightarrow X$ with two global section: one nonvanishing holomorphic $e_1 := \pi^* e_+$ and one meromorphic $e_2:=\pi^* e_-$. Since $\pi$ does not ramify over $L_\infty^0 := (u_0=\infty)$ nor $L_\infty^1 := (u_1=\infty)$, $e_2$ has associated divisor  $ - (L_\infty^0+L_\infty^1)$; thus:
\begin{displaymath}
\Fr \cong \Or_{X} \oplus \Or_{X}(-1,-1) \, .
\end{displaymath}

To better understand the bundle $F$, start by considering the rank two trivial bundle $E_0 := \CC^2 \times X  \rightarrow X$ over $X$; it has two independent (constant) holomorphic global sections $f_1\equiv (1,0)$ and $f_2\equiv (0,1)$. Define the following involution: 
\begin{align*}
 E_0 &  \rightarrow E_0 \\
(u_0,u_1,(Z_1,Z_2)_{e_1,e_2}) & \mapsto (u_1,u_0,(Z_1,-Z_2)_{e_1,e_2}) \; .
\end{align*}

First of all, note that its action on the base coincides with that of the involution $\bar{\eta}$. One can then identify two global invariant sections of the bundle $E_0$: 
\begin{itemize}
\item $f_1$, which is holomorphic;
\item $\hat{f_2}:=b \cdot  f_2$, where $b$ is the global meromorphic function $(u_0,u_1)\mapsto	u_0-u_1$. 
\end{itemize}
The local expression $b \cdot  f_2$ defines a global meromorphic section with associated divisor $\Delta - (L_\infty^0+L_\infty^1)$. The $\Or_X$--module spanned by the sections $f_1$ and $\hat{f_2}$ is isomorphic to the rank two locally free module $\Fr$ (by mapping $f_1$ to $e_1$ and $\hat{f_2}$ to $e_2$) and as such defines a rank two vector bundle over $X$ isomorphic to $F$. More precisely, one goes (locally) from $E_0$ to $F$ using the following transformation (which is trivial on the base): 
\begin{displaymath}
\begin{pmatrix}
Z_1 \\ 
Z_2
\end{pmatrix}_{e_1,e_2} \mapsto  \begin{pmatrix}
Z_1\\ 
\frac{Z_2}{u_0-u_1}
\end{pmatrix}_{e_1,\tilde{e_2}} \, .
\end{displaymath}

\subsection{A rank one projective bundle}\label{sec:Rk1PB} By quotient on the fibres ($\PP(\CC^2) = \PP^1_\CC)$, one can associate to both vector bundles $E_0$ and $E$ rank one projective bundles $\PP^1 \times X$ and $\PP(E)$. We can describe the action of $\eta$ on the former as follows: 
\begin{align*}
\eta : \PP^1 \times \PP^1\times \PP^1 &  \rightarrow \PP^1 \times \PP^1\times \PP^1\\
(u_0,u_1,[Z_1:Z_2]) & \mapsto (u_1,u_0,[Z_2:Z_1]) \quad ;
\end{align*}
or, in the affine chart "$z=\frac{Z_1}{Z_2}$", $(u_0,u_1,z) \mapsto (u_1,u_0,1/z)$.

One goes from $\PP^1 \times X$ to $\PP(E)$ through the following invariant rational functions: 
\begin{displaymath}
\left\lbrace \begin{array}{l}
s = u_0+u_1 \\ 
p = u_0u_1 \\
w := (u_0-u_1) \dfrac{z+1}{z-1}
\end{array} \right. \quad ;
\end{displaymath}
here $(s,p)$ gives us local coordinates over the base and $w$ does the same in the fibres. We will use this projective point of view throughout this paper as it allows for easier computations in the long run. It will also allow us to define an interesting family of Lotka--Volterra foliations in Section \ref{sec:LV}.

\subsection{Logarithmic flat connections} \label{sec:setup}

Start by endowing the trivial rank two vector bundle $E_0\rightarrow X$ with the following logarithmic flat connection :
\begin{displaymath}
\nabla_0 := \dd  + \dfrac{1}{2}\begin{pmatrix}
\omega_0 & 0 \\ 
0 & -\omega_0
\end{pmatrix} \; ,
\end{displaymath}
where  $u_0,u_1$ are projective coordinates on the base $X$ and
\begin{displaymath}
\omega_0 :=\lambda_0 \left( \dfrac{\dd u_0}{u_0} - \dfrac{\dd u_1}{u_1} \right) + \lambda_1   \left( \dfrac{\dd u_0}{u_0-1} - \dfrac{\dd u_1}{u_1-1} \right)  \, ,
\end{displaymath}
with $(\lambda_0, \lambda_1) \in \CC^2\setminus \lbrace (0,0) \rbrace$. This connection has singular locus equal to six lines in $X$ (if $\lambda_0\lambda_1(\lambda_0 + \lambda_1) \neq 0$) and naturally gives rise to a Riccati foliation, defined by the following one--form over $\PP(E_0) = X \times \PP^1$:
\begin{displaymath}
\Ric(\nabla_0) := \dd z + \omega_0 z \, \text{ where $z$ is a projective coordinate on the fibres. }
\end{displaymath}

Moreover one easily checks that $\frac{u_0-u_1}{z}\Ric(\nabla_0)$ is an $\eta$--invariant logarithmic one--form over $\PP(F)$, associated with some connection $\nabla_1$ in the following sense: if one has, in some local chart
\begin{displaymath}
(u_0-u_1)\Ric(\nabla_0) = \dd z + \alpha_2 z^2 + 2\alpha_1 z + \alpha_0\; ,
\end{displaymath}
with $\alpha_0, \alpha_1, \alpha_2$ meromorphic one--forms (remark that here $\alpha_2$ and $\alpha_0$ are zero), then one can set (in the same local chart)
\begin{displaymath}
 \nabla_1 := \dd + \begin{pmatrix}
\alpha_1 & \alpha_0 \\ 
-\alpha_2 & -\alpha_1
\end{pmatrix} \; .
\end{displaymath}
Since the associated Riccati form  $\Ric(\nabla_1):=(u_0-u_1)\Ric(\nabla_0)$ is $\eta$--invariant one can use $\nabla_1$ to get a logarithmic connection $\nabla_2$ on $E$ with poles along $(y=0)$, $(x=0)$,  $L_\infty$ and $\Cr$, the latter coming from the fact that $ \pi$ ramifies there. More precisely, in the affine chart described in Subsection \ref{sec:Rk1PB} one has: \small
\begin{align*}
\Ric(\nabla_2) = \dd w & + \dfrac{1}{2(x^2+y^2+1-2(xy+x+y))}\left( g(\lambda_0, \lambda_1,w,x,y)\dfrac{\dd x}{x} + g(\lambda_1, \lambda_0,w,y,x)\dfrac{\dd y}{y} \right) \; ,
\end{align*}\normalsize
where
\begin{align*}
g(\lambda_0, \lambda_1,w,x,y) =& - ( (2\lambda_0 + \lambda_1)x + \lambda_1(y-1) )w^2 + 2(y-x+1)xw\\
& +(2\lambda_0 + \lambda_1)x^3 - ( (4\lambda_0+\lambda_1)(y+1) + 2 \lambda_1 )x^2\\
&- ( (-2\lambda_0 + \lambda_1)y^2 + 2(2\lambda_0+\lambda_1)y - (2\lambda_0+3\lambda_1) )x\\
& + \lambda_1( y^3 - 3y^2+3y -1)\; .
\end{align*}

\subsection{Trivialisations} \label{sec:TrivialB} We wish to turn $\nabla_2$ into a connection on the trivial bundle $\CC^2 \times \PP^2$; this can be done simply by blowing up the pole of any global meromorphic section of $\PP(E)$ then contracting a suitable divisor, however we want to do so without disturbing the logarithmic nature of the connection $\nabla_2$. 

\begin{lemm}There exists a birational mapping $\Phi : \PP(E) \DashedArrow[->,densely dashed    ] \PP^1 \times \PP^2$ conjugating $\Ric(\nabla_2)$ to some Riccati one--form, that is associated with a logarithmic flat connection $\nabla$ over the trivial bundle $\CC^2 \times \PP^2 \rightarrow \PP^2$.
\end{lemm}
\begin{proof}
First remark that we have the following local expression along $(y=0)$: 
\begin{displaymath}
\Ric(\nabla_2)_{|_{(y=0)}} = \dd w + f(x)(w+x-1)(w-x+1) \dfrac{\dd y}{y} \; .
\end{displaymath}
This tells us that the codimension one foliation associated with the one--form $\Ric(\nabla_2)$ has two singular points on each fibre above $(y=0)$, namely at $w = \pm (x-1)$. So in order to get a birationnal map $\PP(E) \DashedArrow[->,densely dashed    ] \PP^1 \times \PP^2$ one can proceed as follows: 
\begin{itemize}
\item move one of the aforementioned singular loci (e.g $(w=x-1)\cap(y=0)$) at $(w=y=0)$;
\item blow up $(y=0)\cap(w=\infty)$ then contract the strict transform of the fibre at $(y=0)$ on $(w=y=0)$. This latest step is achieved (in our usual affine chart) through the birational map $(x,y,w) \mapsto (x,y,w/y)$.
\end{itemize}
Explicitly in our local chart, the mapping $\Phi$ is given by $$\Phi(w,x,y) = (y(w-x+1),x,y)\; .$$This means that we are blowing up (inside the total space) a line in each fibre over $(y=0)$ then contracting the strict transforms of said fibres thus resolving the singularities of the global meromorphic section $e_-$ described in the proof of Theorem \ref{TheoA}; this shows that our mapping does indeed end in a trivial bundle and since we took care of contracting divisors only on points of the singular locus of the foliation associated with $\Ric(\nabla_2)$ we get a logarithmic flat connection over $\CC^2 \times \PP^2 \rightarrow \PP^2$.
\end{proof}

In the end, one gets a connection $\nabla=\nabla_{\lambda_0, \lambda_1}$ on the trivial bundle $\CC^2 \times \PP^2$ that almost satisfies condition (i) in Theorem \ref{TheoA}, the only thing left to check being whether or not it is a $\mathfrak{sl}_2(\CC)$--connection. Explicitly, the Riccati form associated with $\nabla$ is given by (in the affine chart $(t=1)$):
\begin{align*}
\Ric(\nabla) = \dd w - \dfrac{1}{2(x^2+y^2+1-2(xy+x+y))}\left( f_1(w,x,y)\dfrac{y\dd x}{x} - f_2(w,x,y)\dfrac{x\dd y}{y} \right) \;
\end{align*}
where 
\begin{align*}
 f_1(x,y) = & ( (2 \lambda_0 + \lambda_1)x + \lambda_1(y-1) )y w^2\\
 & + 2( (2\lambda_0+\lambda_1-1)x^2 + ( (\lambda_1+1)y - (2\lambda_0+2\lambda_1-1))x - \lambda_1(y-1) )w\\
& + 2( 2\lambda_0 + \lambda_1 -1)x^2 + ( (-2\lambda_0+\lambda_1+2)y + 2(2\lambda_0-3) )x\\
& + \lambda_1(-y^2+3y-2)
\end{align*} 
and 
\begin{align*}
f_2(x,y) = &  ( \lambda_0(x-1) +(\lambda_0+2\lambda_1)y )xw^2\\
& + 2 ( (\lambda_0-1)(x^2+1)+ ( (\lambda_0+2\lambda_1+1)y - 2 (\lambda_0-1) )x - (\lambda_0+2\lambda_1-1)y) w\\
& + 2(\lambda_0-1)(x^2-1) + (\lambda_0+4\lambda_1+2)yx - (\lambda_0+2\lambda_1)y^2 + (3\lambda_0+4\lambda_1-2 )y \; .
\end{align*}
Note that our birational transformation has "broken" the symmetry between the two components $f_1$ and $f_2$.

We can explicitly compute the residues of the connection $\nabla = \nabla_{\lambda_0, \lambda_1}$ and so check that it is indeed a $\mathfrak{sl}_2(\CC)$--connection (see Table \ref{table:Res_N3}); note that the eigenvalues at $(y=0)$ have been slightly modified because we moved the singular points of the associated foliation.

\begin{table}[!!h]

\begin{center}
\begin{tabular}{c|c|c}
Divisor & Residue & Eigenvalues \\ 
\hline 
& & \\
$y=0$ & $\left(\begin{array}{cc}
-\frac{1}{2} \, \lambda_{0} + \frac{1}{2} & \frac{2 \, {\left(\lambda_{0} - 1\right)}}{y - x + 1} \\
0 & \frac{1}{2} \, \lambda_{0} - \frac{1}{2}
\end{array}\right)
$ & $\pm \dfrac{\lambda_0-1}{2}$\\ 
& & \\
\hline 
& &\\
$x=0$ & $\left(\begin{array}{cc}
-\frac{\lambda_{1} {\left(y - x + 1\right)}}{2 \, {\left(y - x - 1\right)}} & \frac{2 \, \lambda_{1}}{y - x - 1} \\
-\frac{\lambda_{1} {\left(y - x\right)} }{2 \, {\left(y - x - 1\right)}} & \frac{\lambda_{1} {\left(y - x + 1\right)}}{2 \, {\left(y - x - 1\right)}}
\end{array}\right)$& $\pm \dfrac{\lambda_1}{2}$\\ 
& &\\
\hline 
& &\\
$\Cr$ & $\left(\begin{array}{cc}
\frac{{\left(2 \, \lambda_{0} + 2 \, \lambda_{1} - 1\right)} {\left(y - x + 1\right)} - 4 \, \lambda_{0} + 2}{4 \, {\left(y - x - 1\right)}} & -\frac{2 \, {\left({\left(\lambda_{0} + \lambda_{1} - 1\right)} {\left(y - x + 1\right)} - 2 \, \lambda_{0} + 2\right)}}{(y-x+1)(y-x-1)} \\
\frac{{\left(\lambda_{0} + \lambda_{1}\right)} {(y-x+1)(y-x-1)}}{8 \, {\left(y - x - 1\right)}} & -\frac{{\left(2 \, \lambda_{0} + 2 \, \lambda_{1} - 1\right)} {\left(y - x + 1\right)} - 4 \, \lambda_{0} + 2}{4 \, {\left(y - x - 1\right)}}
\end{array}\right)$ &$\pm \dfrac{1}{4}$ \\ 
& &\\
\hline 
& &\\
$L_\infty$  &  $ \left(\begin{array}{cc}
-\frac{1}{2} \, \lambda_{0} - \frac{1}{2} \, \lambda_{1} & 0 \\
\frac{\lambda_{0} + \lambda_{1}}{2 \, {\left(X - 1\right)}} & \frac{1}{2} \, \lambda_{0} + \frac{1}{2} \, \lambda_{1}
\end{array}\right)
$ &$\pm \dfrac{\lambda_0+\lambda_1}{2}$\\ 
($X=x/y$) & &
\end{tabular} 
\end{center}
\caption{Residues for $\nabla$.}
\label{table:Res_N3}
\end{table}

\subsection{Monodromy representation}\label{sec:MonodRep}

To conclude the proof of Theorem \ref{TheoA} one needs to compute the monodromy representation of the connection $\nabla$ and see that it is, as announced, a dihedral representation of $\Gamma$ into $SL_2(\CC)$. First, let us prove a result announced in Subsection \ref{sec:topology}.

\renewcommand{\thepropB}{\ref{prop:repres}}
\begin{propB}
The only (up to conjugacy) family of non--degenerate representations of $\Gamma$ into $SL_2(\CC)$ is as follows:
$$\rho_{u,v} : a  \mapsto \begin{pmatrix}
0 & 1 \\ 
-1 & 0
\end{pmatrix}  , \quad 
b  \mapsto \begin{pmatrix}
u & 0 \\ 
0 & u^{-1}
\end{pmatrix}, \quad
c  \mapsto \begin{pmatrix}
v & 0 \\ 
0 & v^{-1}
\end{pmatrix} \; , \text{ for }  u,v \in \CC^* \;.$$
\end{propB}

\begin{proof}
Let $\rho$ be such a representation; since $\mathrm{Im}(\rho)$ must be non--abelian, either $C:=\rho(c)$ or $B:=\rho(b)$ does not commute to $A:=\rho(a)$, say $B$. Then $(AB)^2 = (BA)^2$ and so $(AB)^2$ commutes to the non--abelian subgroup spanned by $A$ and $B$ in $PSL_2(\CC)$, therefore $(AB)^2$ must be equal to $\varepsilon I_2$ for some $\varepsilon \in \lbrace -1,1 \rbrace$. This means that $AB$ is diagonisable with eigenvalues in either $\lbrace -1 , 1 \rbrace$ (if $\varepsilon = 1$) or  $\lbrace -i , i \rbrace$. In the former case, $AB$ would be equal to $\pm I_2$ and so one would have $AB = BA$. Therefore, $(AB)^2$ must be equal to $-I_2$.

Up to conjugacy, one can assume that $A$, $B$ and $C$ are of the form
\begin{displaymath}
A = \begin{pmatrix}
\alpha & \beta \\ 
-1 & \gamma
\end{pmatrix} \quad , \quad B= \begin{pmatrix}
\mu & \kappa \\ 
0 & \mu^{-1}
\end{pmatrix} \quad \text{and} \quad C= \begin{pmatrix}
\tau & \chi \\ 
0 & \tau^{-1}
\end{pmatrix} \; .
\end{displaymath}
Since $(AB)^2 = - I_2$, $AB = - B^{-1}A^{-1}$ and so one must have 
\begin{equation} 
\left \lbrace \begin{array}{ccc}
\alpha \gamma + \beta &=&\det(A)=1 \\
\alpha \mu &=& \gamma \mu^{-1} + \kappa
\end{array} \right. \; .
\end{equation}
We assumed the monodromy representation to be non--degenerate; this means in particular that $ABAC=-B^{-1}A^{-1}AC = -B^{-1}C$ must not be equal to $\pm I_2$, i.e $B \neq \pm C$ and so $\mu^2 \neq \tau^2$.

\textbf{Case 1: $\mu^2 \neq 1$.} In this case, $B$ is diagonalisable so it is possible (up to conjugacy) to assume $\kappa = 0$. Since $B$ commutes to $C$, it follows that $\chi$ must also be zero and $\tau^2 \neq  1$. This implies that $A$ does not commute to $C$ and so one gets
\begin{equation} 
\left \lbrace \begin{array}{ccc} 
\alpha \gamma + \beta &=&1 \\
\alpha \mu^2 &=& \gamma \\
\alpha \tau^2 &=& \gamma \\
\end{array} \right. \; .
\end{equation}
As $\tau^2 \neq \mu^2$, this forces $\alpha$ and $\gamma$ to be zero, thus $\beta$ must be one.

\textbf{Case 2: $\mu^2=1$.} Since $B$ is not projectively trivial, then $\kappa$ must be nonzero. The fact that $B$ must commute to $C$ forces $\tau^2$ to be one and so one must also have $\chi \neq 0$. It is therefore impossible for $A$ to commute to $C$ and so by a similar reasoning to the one above, $(AC)^2 = - I_2$, thus one gets
\begin{equation} 
\left \lbrace \begin{array}{ccc}
\alpha \gamma + \beta &=&\det(A)=1 \\
\alpha \mu &=& \gamma \mu^{-1} + \kappa\\
\alpha \tau &=& \gamma \tau^{-1} + \chi
\end{array} \right. \; ,
\end{equation}
which is equivalent to (since $\mu^2 = \tau^2 = 1$)
\begin{equation} 
\left \lbrace \begin{array}{ccc}
\alpha \gamma + \beta &=&\det(A)=1 \\
\alpha  &=& \gamma  + \kappa\mu\\
\alpha  &=& \gamma  + \chi\tau
\end{array} \right. \; ,
\end{equation}
therefore $\kappa\mu=\chi\tau$. This means that $B = \pm C$ and so $ABAC = \pm I_2$, which contradict the non--degeneracy condition.
\end{proof}

\begin{rema}
This implies that any non--degenerate representation of $\Gamma$ will have a "sizeable" kernel; indeed recall that we have a natural two--fold ramified covering $\pi : \PP^1 \times \PP^1 \xrightarrow[]{2:1} \PP^2$ ramifying over the diagonal $\Delta$. This mapping yields a nonramified covering $\tilde{\pi} : X-D \xrightarrow[]{2:1} \PP^2-\Qr$ and thus one gets that $\pi_1(X-D)$ embeds into $\pi_1(\PP^2 - \Qr) \cong \Gamma$ as an index two subgroup. If one denotes by $\PP^1_n$ the $n$ punctured sphere, the projection on the line $(y=0)$ gives a fibration $X-D \rightarrow \PP^1_3$ with fibre $\PP^1_4$. As the universal covering of $\PP^1_3$ (namely the hyperbolic plane $\HH^2$) is contractible, the homotopy exact sequence associated with this fibration yields:
\begin{displaymath}
0 = \pi_2(\PP^1_3) \rightarrow \pi_1(\PP^1_4) \rightarrow \pi_1(X-D) \rightarrow \pi_1(\PP^1_3) \rightarrow 0\; .
\end{displaymath}
In particular, there is an injective morphism from $\pi_1(\PP^1_4) \cong \Fb_3$, where $\Fb_r$ denotes the free group over $r$ generators, into $\pi_1(X-D)$; which in turn implies that the group $\Gamma$ contains a noncommutative free group.

Moreover the orbifold fundamental group $\Gamma_\pi^{\mathrm{orb}}$ associated with the ramified covering $\pi$ also contains a free group; more precisely if we define 
\begin{displaymath}
\Gamma_\pi^{\mathrm{orb}} := \langle a,b,c \, | \, (ab)^2 (ba)^{-2}= (ac)^2 (ca)^{-2}= [b,c]= a^2=1 \rangle \;  
\end{displaymath}
then $\pi$ induces an embedding of the fundamental group of $X$ minus six lines into $\Gamma_\pi^{\mathrm{orb}}$, i.e $\Fb_2  \times \Fb_2 \hookrightarrow \Gamma_\pi^{\mathrm{orb}}$. This is especially relevant since the projective representations $\Gamma \rightarrow PSL_2(\CC)$ associated with the monodromy of the connections we will describe in this paper factor through this orbifold fundamental group.
\end{rema}

Now we will describe the local monodromy around irreducible components of the polar locus. Let $C$ be an irreducible curve contained in the polar locus of some logarithmic flat $\mathfrak{sl}_2(\CC)$--connection $\nabla$ over $\PP^2$, with associated monodromy representation $\varrho$. Set a point $p \in C$ such that no other irreducible curve in the polar locus of the connection passes through $p$; then if $U$ is a sufficiently small analytic neighbourhood of $p$ one gets:
\begin{displaymath}
\pi_1(U \setminus C \cap U) \cong \ZZ \;.
\end{displaymath}
Let $\gamma$ be any loop generating the above cyclic group; the conjugacy class of the matrix $\varrho(\gamma)$ does not depend on the choice of a base point for the fundamental group. Indeed, if $\gamma$ is chosen as above for some base point $q$ and if $q^\prime$ is some other point in the complement of the polar locus, then if one takes $\delta$ to be any path between $q^\prime$ and $q$, the loop $\delta \cdot \gamma \cdot \delta^{-1}$ is an element of the fundamental group of the complement based at $q^\prime$ whose monodromy is conjugate to $\varrho (\gamma)$.

\begin{defi}\label{def:locMon}
Using the notations above, define the local monodromy of $\nabla$ around $C$ as the conjugacy class of the matrix $\varrho (\gamma)$.
\end{defi}

For $j=0,1$ set $a_j := e^{-i\pi \lambda_j}$; the monodromy associated with the connection $\nabla_0$ is as follows:
\begin{itemize}
\item around $u_0=j$ (resp. $u_1=j$), $j=0,1$, it is the multiplication by $a_j$ (resp. $a_j^{-1}$); 
\item around $u_0=\infty$ (resp. $u_1=\infty$), it is the multiplication by $a_0^{-1}a_1^{-1}$ (resp. $a_0a_1$).
\end{itemize}
This is a complete description since the fundamental group of the projective line minus six lines is isomorphic to $\Fb_2\times \Fb_2$ and is generated by loops going around $x,y = 0,1$ once.

The monodromy of the connection $\nabla_2$ comes directly from that of $\nabla_0$ around the three lines in its singular locus; more precisely we can explicitly compute (up to conjugacy) its local monodromy around:
\begin{itemize}
\item $(y=0)$: 
\begin{displaymath}
 \begin{pmatrix}
a_0 & 0 \\ 
0 & a_0^{-1}
\end{pmatrix}  \; ;
\end{displaymath} 
\item $(x=0)$: 
\begin{displaymath}
 \begin{pmatrix}
a_1 & 0 \\ 
0 & a_1^{-1}
\end{pmatrix}  \; ;
\end{displaymath} 
\item and $L_\infty$: 
\begin{displaymath}
 \begin{pmatrix}
(a_0a_1)^{-1} & 0 \\ 
0 & a_0a_1
\end{pmatrix}  \; .
\end{displaymath} 
\end{itemize}

However the monodromy of $\nabla_2$ around the conic $\Cr$ comes solely from the ramification of the covering $\pi$. More precisely since any path linking $(u_0,u_1) \in X$ to $(u_1,u_0)$ pushes back as a loop on the quotient $\PP^2 = \pi(X)$ and since any local solution $z$ of $\nabla_0$ satisfies $z(u_1,u_0) = \dfrac{1}{z(u_0,u_1)}$ the monodromy group of $\nabla_2$ must contain the following matrix:
\begin{displaymath}
\begin{pmatrix}
0 & 1 \\ 
1 & 0
\end{pmatrix} \; .
\end{displaymath}

\begin{prop}\label{prop:globMonod}
The monodromy group of the connection $\nabla$ is the subgroup of the infinite dihedral group
\begin{displaymath}
\mathbf{D}_\infty := \left\lbrace \left. \begin{pmatrix}
0 & \alpha \\ 
-\alpha^{-1} & 0
\end{pmatrix}\, , \: \begin{pmatrix}
 \beta & 0 \\ 
0& \beta^{-1}
\end{pmatrix} \, \right| \, \alpha,\beta \in \CC^* \right\rbrace \leq SL_2(\CC)\, 
\end{displaymath}
generated by the following three matrices:
\begin{displaymath}
\begin{pmatrix}
0 & 1 \\ 
-1 & 0
\end{pmatrix}, \quad \begin{pmatrix}
-e^{-i \pi \lambda_0} & 0 \\ 
0 & -e^{i \pi \lambda_0}
\end{pmatrix} \quad \text{ and }  \begin{pmatrix}
e^{-i \pi \lambda_1} & 0 \\ 
0 & e^{i \pi \lambda_1}
\end{pmatrix} \; .
\end{displaymath}
\end{prop}
\begin{proof}
We know from Proposition \ref{th:Degtyarev} that the fundamental group of the complement of the singular locus of $\nabla$ in $\PP^2$ has the following presentation:
\begin{displaymath}
\Gamma = \langle a,b,c \, | \, (ab)^2 (ba)^{-2}= (ac)^2 (ca)^{-2}= [b,c]=1 \rangle \; ;
\end{displaymath}
and that we can take $a$ to be a loop whose lift is some path in $X$ joining $(x,y)$ and $(y,x)$ (for generic $(x,y) \in X$) and $b$ (resp. $c$) to be a loop going around $(y=0)$ (resp. $(x=0)$) once (see Fig. \ref{fig:Pi_1_gen}). If we choose a set of local coordinates in which the monodromy matrices of both $b$ and $c$ are diagonal (this is possible because the two loops commute) then the monodromy of $a$ only comes from the covering $\pi$ and is equal to:   $$\begin{pmatrix}
0 & 1 \\ 
-1 & 0
\end{pmatrix} \, . $$
In conclusion, the monodromy representation is given by the following matrices:
\begin{displaymath}
\begin{pmatrix}
0 & 1 \\ 
-1 & 0
\end{pmatrix}, \quad \begin{pmatrix}
-e^{-i \pi \lambda_0} & 0 \\ 
0 & -e^{i \pi \lambda_0}
\end{pmatrix} \quad \text{ and }  \begin{pmatrix}
e^{-i \pi \lambda_1} & 0 \\ 
0 & e^{i \pi \lambda_1}
\end{pmatrix} \; ,
\end{displaymath}
which are elements of $\mathbf{D}_\infty$. 
\end{proof}

\section{Algebraic Garnier solutions}

In this section we show that the connection $\nabla$ induces an isomonodromic deformation over the four and five punctured spheres. Furthermore we give rational parametrisations of the associated  algebraic Painlevé VI and Garnier solutions and a description of the associated monodromy representation.

\subsection{Painlevé VI solutions}\label{sec:PVI} 

It is well known \cite{Jimbo,Hitchin} that isomonodromic deformations of rank two $\mathfrak{sl}_2(\CC)$--connections over the four punctured sphere correspond to solutions of the sixth Painlevé equation, namely the following order two nonlinear differential equation:
\begin{align*}
\dfrac{d^2q}{du^2} = &\dfrac{1}{2} \left( \dfrac{1}{q} + \dfrac{1}{q-1} + \dfrac{1}{q-u}\right) \left( \dfrac{\dd q}{du} \right)^2  \\
& - \left( \dfrac{1}{u} + \dfrac{1}{u-1} + \dfrac{1}{q-u} \right)\dfrac{\dd q}{du}\\
& + \dfrac{q(q-1)(q-u)}{u^2(u-1)^2} \left( \alpha + \beta \dfrac{u}{q^2} + \gamma \dfrac{u-1}{(q-1)^2}+ \delta \dfrac{u(u-1)}{(q-u)^2}\right)\, ,
\end{align*}
where $\alpha, \beta,\gamma$ and $\delta$ are complex--valued parameters.

\begin{figure}
\begin{center}
\begin{pspicture}(0,-3.85)(11.3,3.85)
\definecolor{color3}{rgb}{0.00392156862745098,0.00392156862745098,0.00392156862745098}
\definecolor{color18}{rgb}{0.0196078431372549,0.611764705882353,0.00392156862745098}
\pscircle[linewidth=0.04,linecolor=red,dimen=outer](6.37,-0.28){1.79}
\psline[linewidth=0.04cm,linecolor=color3](0.0,-3.83)(7.5,3.23)
\psline[linewidth=0.04cm,linecolor=color3](4.16,3.83)(10.96,-1.71)
\psline[linewidth=0.04cm,linecolor=color3](0.04,-3.47)(11.0,-1.11)
\psline[linewidth=0.04cm,linecolor=color18](11.12,-1.33)(1.94,-0.25)
\psline[linewidth=0.04cm,linecolor=color18](11.28,-1.53)(1.84,1.39)
\psline[linewidth=0.04cm,linecolor=color18](11.0,-1.23)(1.12,-1.39)
\psline[linewidth=0.04cm,linecolor=color18](11.04,-1.63)(4.34,2.31)
\usefont{T1}{ptm}{m}{n}
\rput(4.6,-1.9){\color{red}$\Cr$}
\end{pspicture} 
\end{center}

\caption{Special lines.}\label{fig:specialLines}
\end{figure}
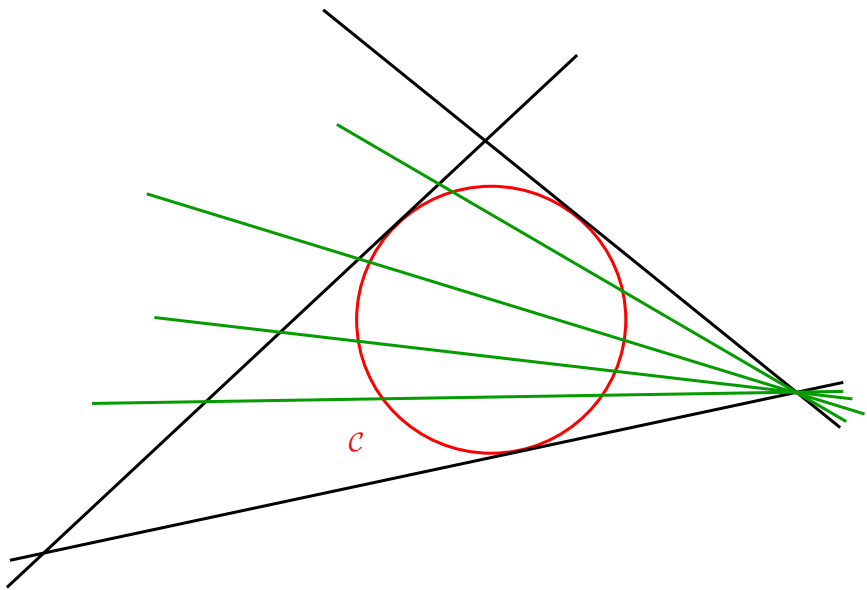

Let us look at the connection induced by $\nabla$ on the family of lines going through $P_0 := (x=0)\cap L_\infty$ (see Fig. \ref{fig:specialLines}) that are neither $(x=0)$ nor the line at infinity; these are the lines of the form $(y=c)$ in the affine chart $(x,y)$ from Subsection \ref{sec:Rk2VB}. According to Subsection \ref{sec:TrivialB}, this corresponds to studying the isomonodromic  deformation given by the following Riccati forms, for generic $y$: 
\begin{displaymath}
\Ric(\nabla_y) := \dd w - \dfrac{y}{2x(x^2+y^2+1-2(xy+x+y))} f_y(x,w)\dd x \: ,
\end{displaymath}
where 
\begin{align*}
 f_y(x,w) = & ( \lambda_0(x-1) +(\lambda_0+2\lambda_1)y )yw^2\\
& + 2 ( (\lambda_0-1)(x^2+1)+ ( (\lambda_0+2\lambda_1+1)y - 2 (\lambda_0-1) )x - (\lambda_0+2\lambda_1-1)y) w\\
& + 2(\lambda_0-1)(x^2-1) + (\lambda_0+4\lambda_1+2)yx - (\lambda_0+2\lambda_1)y^2 + (3\lambda_0+4\lambda_1-2 )y \; .
\end{align*} 
From this isomonodromic deformation we produce algebraic solutions of the Painlevé VI equation by adapting part of a paper by Hitchin \cite{Hitchin}.

\begin{prop}
The family of algebraic solutions of the Painlevé VI equation associated with the connections $(\nabla_{\lambda_0,  \lambda_1})_{\lambda_0,  \lambda_1}$ is given by the functions
\begin{displaymath}
q(u) = -\dfrac{\lambda_1}{2\lambda_0 + \lambda_1} \sqrt{u}  
\end{displaymath} 
and the parameters:
$$
\alpha = \dfrac{(2\lambda_0+\lambda_1)^2}{2} \; ,\beta  = - \dfrac{\lambda_1^2}{2} \; , \gamma = 1/8 \text{ and }\delta = 3/8  \; .
$$
\end{prop}

\begin{proof}
Let $z$ be a parameter such that $z^2 = y$; then $\nabla_y$ has poles at $x = (z\pm 1)^2$, $x=0$ and $x = \infty$. Up to Möbius transformation, one can assume that these are in fact located at $s \in \lbrace 0,1,u(z),\infty \rbrace $, with: 
\begin{displaymath}
u(z) = \dfrac{z^2-2z+1}{z^2+2z+1} = \dfrac{(z-1)^2}{(z+1)^2} \, .
\end{displaymath}
It is then possible to compute the relevant data associated with this family of connections (see Table \ref{table:Res_Np}).

\begin{table}
\begin{center}
\begin{tabular}{c|c|c}
Pole & Residue & Eigenvalues \\ 
\hline 
& & \\
$x=0$ & $W_0:=\left(\begin{array}{cc}
-\frac{\lambda_{1} (z^{2} + 1)}{2 \, {\left(z^{2} - 1\right)}} & \frac{2 \, \lambda_{1}}{z^{2} - 1} \\
-\frac{\lambda_{1} z^{2}}{2 \, {\left(z^{2} - 1\right)}} & \frac{\lambda_{1}( z^{2} + 1)}{2 \, {\left(z^{2} - 1\right)}}
\end{array}\right)
$ &$\pm \dfrac{\lambda_1}{2}$ \\ 
& &  \\
\hline 
& &\\
$x=1$ &$W_1:= \left(\begin{array}{cc}
\frac{{\left(2 \, \lambda_{0} + 2 \, \lambda_{1} - 1\right)} z + 2 \, \lambda_{0} - 1}{4 \, {\left(z + 1\right)}} & \frac{{\left(\lambda_{0} + \lambda_{1} - 1\right)} z + \lambda_{0} - 1}{z^{2} + z} \\
-\frac{{\left(\lambda_{0} + \lambda_{1}\right)} z^{2} + \lambda_{0} z}{4 \, {\left(z + 1\right)}} & -\frac{{\left(2 \, \lambda_{0} + 2 \, \lambda_{1} - 1\right)} z + 2 \, \lambda_{0} - 1}{4 \, {\left(z + 1\right)}}
\end{array}\right)
$  & $\pm \dfrac{1}{4} $\\ 
& &\\
\hline 
& &\\
$x=t(z)$ & $ W_2:=\left(\begin{array}{cc}
\frac{{\left(2 \, \lambda_{0} + 2 \, \lambda_{1} - 1\right)} z - 2 \, \lambda_{0} + 1}{4 \, {\left(z - 1\right)}} & -\frac{{\left(\lambda_{0} + \lambda_{1} - 1\right)} z - \lambda_{0} + 1}{z^{2} - z} \\
\frac{{\left(\lambda_{0} + \lambda_{1}\right)} z^{2} - \lambda_{0} z}{4 \, {\left(z - 1\right)}} & -\frac{{\left(2 \, \lambda_{0} + 2 \, \lambda_{1} - 1\right)} z - 2 \, \lambda_{0} + 1}{4 \, {\left(z - 1\right)}}
\end{array}\right)
$ &$\pm \dfrac{1}{4}$ \\ 
& &\\
\hline 
& &\\
$x=\infty$  & $ W:=\left(\begin{array}{cc}
-\lambda_{0} - \frac{1}{2} \, \lambda_{1} + \frac{1}{2} & 0 \\
0 & \lambda_{0} + \frac{1}{2} \, \lambda_{1} - \frac{1}{2}
\end{array}\right)
$  &$\pm \dfrac{2\lambda_0 + \lambda_1-1}{2}$\\ 
 & &
\end{tabular} 
\end{center}
\caption{Residues for $\nabla_{y}$.}
\label{table:Res_Np}
\end{table}

Let us now set $$H := \dfrac{W_0}{x} + \dfrac{W_1}{x-1}+\dfrac{W_2}{x-t}\;,$$ where the $W_i$ are the residues from Table \ref{table:Res_Np}; then since $-W$ is diagonal and equal to the sum $W_0+W_1+W_2$, its lower left coefficient is a degree one polynomial in $x$, whose root can be explicitly computed as a rational function of $z$: 
\begin{displaymath}
q(z) := -\dfrac{\lambda_1}{2\lambda_0 + \lambda_1} \dfrac{z-1}{z+1} \; ,
\end{displaymath}
or as an algebraic function of $u$: 
\begin{displaymath}
q(u) = -\dfrac{\lambda_1}{2\lambda_0 + \lambda_1} \sqrt{u}  \, .
\end{displaymath}
One can then check that this function $u\mapsto q(u)$ is indeed a solution of the sixth Painlevé equation for the announced choice of parameters. 

\end{proof}

\subsection{Restriction to generic lines}\label{sec:ResGenL} 

Let us now consider the connection induced by $\nabla$ on generic lines in $\PP^2$, such a line being given in our usual affine chart by an equation of the form $y=\alpha x+\beta$. We thus obtain an isomonodromic deformation $(\nabla_{\alpha,\beta})_{\alpha,\beta}$ over the five punctured sphere; more precisely if one chooses a parameter $z$ such that $z^2 = \beta(1-\alpha)+\alpha$ then one gets (after Möbius transformation) a family of logarithmic flat connections over $\PP^1 \setminus \lbrace 0,1,t_1,t_2,\infty \rbrace$, where:
\begin{displaymath}
t_1 = -\dfrac{\alpha(z+1)^2}{(\alpha-1)(\alpha-{z}^2)} \quad \text{ and } \quad t_2 = -\dfrac{\alpha(z-1)^2}{(\alpha-1)(\alpha-{z}^2)} \; .
\end{displaymath}

The associated Riccati forms are given by:
\begin{displaymath}
\Ric(\nabla_{\alpha,\beta}) = \dd w +  \dfrac{a_2(x)w^2+ a_1(x)w + a_0(x) }{2x(x-1)(x-t_1)(x-t_2)}\dd x \; 
\end{displaymath}
where: 
\begin{align*}
\dfrac{a_2(x)}{\alpha(x-1)({z}^2-\alpha)} =& (\lambda_0+\lambda_1)( \alpha^2 -(z^2+1)\alpha + z^2 )x^2 \\
&+(-\lambda_1\alpha^2+( \lambda_0(z^2+1)+2\lambda_1)\alpha - (2\lambda_0 +\lambda_1)z^2 )x\\
&+\lambda_1(z^2-1)\alpha \\
\end{align*}
\begin{align*}
\dfrac{a_1(x)}{2} =& (\lambda_0+\lambda_1)(\alpha^4-2(z^2+1)\alpha^3+( z^4+4z^2+1)\alpha^2-2(z^4+z^2)\alpha+z^4)x^3\\
& + [ -(2\lambda_0+3\lambda_1-1)\alpha^4 \\
& \qquad +( (4\lambda_0 + 4\lambda_1-1)z^2+4\lambda_0+6\lambda_1-1 )\alpha^3 \\
& \qquad- ((2\lambda_0+\lambda_1)z^4 +2(4\lambda_0+\lambda_1-1)z^2 + (2\lambda_0+3\lambda_1))\alpha^2 \\
& \qquad+  ((4\lambda_0+2\lambda_1-1)z^4+(4\lambda_0+4\lambda_1-1)z^2)\alpha \\
& \qquad -(2\lambda_0+\lambda_1+1)z^4 ]x^2\\
& +[ 2\lambda_1 a^4-((2\lambda_0-1)z^2+(2\lambda_0+6\lambda_1-1))\alpha^3 \\
& \qquad +((\lambda_0-\lambda_1)z^4+2(3\lambda_0+2\lambda_1-1)z^2 + \lambda_0+3\lambda_1)\alpha^2\\
& \qquad +((2\lambda_0-1)z^4+(2\lambda_0+2\lambda_1-1)z^2)\alpha ]x \\
& +\lambda_1( 2(1-z^2)\alpha+z^4-1 )\alpha^2
\end{align*}
and
\begin{align*}
\dfrac{a_0(x)}{4\alpha(\alpha-1)} =& (\lambda_0+\lambda_1-1) (1-\alpha)(z^2-\alpha) x^2\\
&+( ( (\lambda_0-1)(\alpha-2)-\lambda_1)z^2 -\lambda_1\alpha^2+(\lambda_0+2\lambda_1-1)\alpha)x \\
&+\lambda_1 \alpha (z^2-1) \; .
\end{align*}
Using the explicit formulas given in Subsection \ref{sec:TrivialB}, we can explicitly compute the spectral data associated with these connections (see Table \ref{table:Res_Nab}). To mirror what we did in Subsection \ref{sec:PVI}, let us assume (up to a change of basis) that the residue at infinity $M$ is diagonal and set:
\begin{displaymath}
\hat{H}:=\dfrac{M_0}{x} + \dfrac{M_1}{x-1} + \dfrac{M_{t_1}}{x-t_1}+ \dfrac{M_{t_2}}{x-t_2} \;;
\end{displaymath}
then since $M$ does not depend on $x$, the lower left coefficient of $\hat{H}$ must be a degree two polynomial in $x$, say:
\begin{equation} \label{coeff}
\hat{H}_{2,1}=\dfrac{c(t_1,t_2)(x^2-S_q(t_1,t_2)x+P_q(t_1,t_2))}{x(x-1)(x-t_1)(x-t_2)} \, ,
\end{equation}
where $S_q := q_1+q_2$ and $P_q := q_1 q_2$, with $q_1,q_2$ some algebraic functions of $(t_1,t_2)$.

\begin{table}
\begin{center}
\begin{tabular}{c|c|c}
Pole & Residue & Eigenvalues \\ 
\hline 
& & \\
$x=0$ & $ M_0:= \left(\begin{array}{rr}
-\frac{\lambda_{1} (z^2-2\alpha+1) }{2 \, {\left({z}^{2} - 1\right)}} & -\frac{2\lambda_{1}  {\left(\alpha - 1\right)} }{{z}^{2} - 1} \\
-\frac{\lambda_{1} ({z}^{2} - \alpha) }{2 \, {\left({z}^{2} - 1\right)}} & \frac{\lambda_{1} (z^2-2\alpha+1)} {2 \, {\left({z}^{2} - 1\right)}}
\end{array}\right)
$ & $ \pm\dfrac{\lambda_1}{2} $ \\ 
& &  \\
\hline 
& &\\
$x=1$ &$M_1:= \left(\begin{array}{cc}
-\frac{1}{2} \, \lambda_{0} + \frac{1}{2} & -\frac{2\alpha(\alpha-1)(1-\lambda_0)}{\alpha^{2} - {z}^{2}} \\
0 & \frac{1}{2} \, \lambda_{0} - \frac{1}{2}
\end{array}\right)
$  & $ \pm\dfrac{1}{2}(\lambda_0-1) $\\ 
& &\\
\hline 
& &\\
$x=t_1$ & $ M_{t_1} :=  \left(\begin{array}{cc}
\frac{(2\lambda_0-1)({z}+1) + 2 \lambda_1(\alpha+{z})}{4 \, {\left({z} + 1\right)}} & -\frac{(\alpha-1)(\lambda_0+\lambda_1\alpha - 1 + (\lambda_0+\lambda_1-1)z)}{{\left(\alpha + 1\right)} {z} + {z}^{2} + \alpha} \\
\frac{\lambda_0({z}^2+(1+\alpha){z}+\alpha) + \lambda_1({z}+\alpha)^2}{4 \, {\left({\left(\alpha - 1\right)} {z} + \alpha - 1\right)}} & -\frac{(2\lambda_0-1)({z}+1) + 2 \lambda_1(\alpha+{z})}{4 \, {\left({z} + 1\right)}}
\end{array}\right)
$ &$\pm \dfrac{1}{4}$ \\ 
& &\\
\hline 
& &\\
$x=t_2$ & $  M_{t_2} := \left(\begin{array}{cc}
-\frac{(2\lambda_0-1)(1-{z}) + 2 \lambda_1(\alpha-t)}{4 \, {\left({z} - 1\right)}} & \frac{(\alpha-1)(\lambda_0+\alpha\lambda_1 - 1 - (\lambda_0+\lambda_1-1)z)}{{\left(\alpha + 1\right)} {z} - {z}^{2} - \alpha} \\
-\frac{\lambda_0({z}^2+(1+\alpha){z}+\alpha) + \lambda_1({z}+\alpha)^2}{4 \, {\left({\left(\alpha - 1\right)} {z} - \alpha + 1\right)}} & \frac{(2\lambda_0-1)(1-{z}) + 2 \lambda_1(\alpha-{z})}{4 \, {\left({z} - 1\right)}}
\end{array}\right)
 $ &$\pm \dfrac{1}{4}$ \\ 
& &\\
\hline 
& &\\
$x=\infty$  & $M:= \left(\begin{array}{rr}
-\frac{1}{2} \, \lambda_{0} - \frac{1}{2} \, \lambda_{1} & 0 \\
- \frac{(\lambda_{0} +  \lambda_{1})\alpha}{2 \, {\left(\alpha - 1\right)}} & \frac{1}{2} \, \lambda_{0} + \frac{1}{2} \, \lambda_{1}
\end{array}\right)$  &$ \pm \dfrac{1}{2}( \lambda_0+\lambda_1 )$\\ 
 & &
\end{tabular} 
\end{center}
\caption{Residues for $\nabla_{\alpha,\beta}$.}
\label{table:Res_Nab}
\end{table}

\subsection{Rational parametrisations}

First remark that one can rewrite (\ref{coeff}) as follows:
$$
x(x-1)(x^2 - S_t x + P_t) \hat{H}_{2,1}= c(t_1,t_2) (x^2 - S_q x + P_q) \, ,$$
where $S_t = t_1+t_2$ and $P_t = t_1 t_2$ are the elementary symmetric polynomials in $(t_1,t_2)$. 

\begin{lemm}\label{lem:RatPar}
The parameters $(\alpha,{z})$ introduced in Subsection \ref{sec:ResGenL} give a rational mapping $(\PP^1)^2 \DashedArrow[->,densely dashed    ](\PP^1)^4$ giving explicit expressions of $(t_1,t_2,S_q,P_q)$, namely: 
\begin{align*}
t_1 & = -\dfrac{\alpha({z}+1)^2}{(\alpha-1)(\alpha-{z}^2)} \, ,\\
t_2 & = -\dfrac{\alpha({z}-1)^2}{(\alpha-1)(\alpha-{z}^2)} \, ,\\
S_q  & = \dfrac{\lambda_0(\alpha^2-2\alpha+z^2)-\lambda_1(1+z^2+2\alpha)\alpha+\alpha(2-\alpha)-z^2}{(\lambda_0+\lambda_1-1)(\alpha-{z}^2)(\alpha-1)} \, ,\\
P_q  & = \dfrac{(\lambda_0-1)({z}-1)({z}+1)\alpha}{(\lambda_0+\lambda_1-1)(\alpha-{z}^2)(\alpha-1)} \; .
\end{align*}
\end{lemm}
\begin{proof}
Using Gröbner bases to eliminate the variable $x$ one obtains a system of equations of the following form: 
\begin{equation}
\left\lbrace \begin{array}{l}
(\lambda_0-1)^2 \lambda_1^2 S_t = -F(S_q,P_q) \\ 
(\lambda_0-1)^2 P_t = - (\lambda_0+\lambda_1-1)^2P_q^2
\end{array} \right. \; ;
\end{equation}
where:
\begin{align*}
F(S_q,P_q) = &(\lambda_0-\lambda_1-1)(\lambda_0+\lambda_1-1)^3 P_q^2\\
& + (\lambda_0-1)^2(\lambda_0+\lambda_1-1)^2(2P_q-2P_qS_q+S_q^2-2Sq)\\
&+(\lambda_0-1)^3(\lambda_0+2\lambda_1-1) \; .
\end{align*}
The discriminant $\Delta_t$ of this system vanishes along $2$ pairs of parallel lines in $\PP^1_{S_q} \times \PP^1_{P_q}$; namely:
\begin{displaymath}
(\Delta_t = 0) = (\alpha=0) \cup (x^\prime=0) \cup (\alpha = \infty) \cup (x^\prime=\infty) \subset	\PP^1_\alpha \times \PP^1_{x^\prime} 
\end{displaymath}
for some projective coordinate $x^\prime$ such that ${z}^2 = \alpha x^\prime$. This explicit description of the two--fold ramified covering given by ${z}$ allows us to parametrize $(S_q,P_q)$ as rational functions of $(\alpha,z)$, hence concluding the proof.
\end{proof}

We can now prove that we have indeed constructed a family of algebraic solutions for a Garnier system. More precisely, consider the following Hamiltonian system:
\begin{equation} \label{Garnier2}
\left\lbrace \begin{array}{ccc}
\dr_{t_k} \mathbf{p}_i & = - \dr_{\mathbf{q}_i} H_k  & i,k = 1,2\\ 
\dr_{t_k} \mathbf{q}_i &  = \dr_{\mathbf{p}_i} H_k & i,k = 1,2
\end{array} \right. \, ,
\end{equation}
where:
\begin{align*}
H_k :=  (-1)^{k}\dfrac{2 H(t_k,t_{3-k},\mathbf{p}_1,\mathbf{p}_2,\mathbf{q}_1,\mathbf{q}_2) + H(t_k,t_{3-k},\mathbf{p}_2,\mathbf{p}_1,\mathbf{q}_2,\mathbf{q}_1) }{2(\mathbf{q}_1-\mathbf{q}_2)(t_1-t_2)(t_k-1)t_k}
\end{align*}
with:\begin{align*}
\dfrac{H(t_1,t_2,\mathbf{p}_1,\mathbf{p}_2, \mathbf{q}_1,\mathbf{q}_2)}{\mathbf{p}_1\mathbf{q}_1(\mathbf{q}_2-t_1)} \; = \; &  \mathbf{p}_1 \mathbf{q}_1^3 + ( (t_1+t_2+1)\mathbf{p}_1 + (\lambda_0+\lambda_1-1) ) \mathbf{q}_1^2\\
& - (  (t_1+t_2+t_1t_2)\mathbf{p}_1 - (2\lambda_0+2\lambda_1-1)(t_1+t_2) - 2t_2 + 2(\lambda_0-1) )\dfrac{\mathbf{q}_1}{2} \\
& + ( -(2\lambda_0-1) t_1t_2\mathbf{p}_1 + 2(\lambda_0+\lambda_1-1)t_2 + 2 \lambda_0 - 1) t_1 + 2(\lambda_0-3)t_2 \; .
\end{align*}

\begin{prop}\label{prop:Garnier}
Let $q_1, q_2$ be the algebraic functions defined in Subsection \ref{sec:ResGenL}; then there exist two algebraic functions $p_1(t_1,t_2)$ and $p_2(t_1,t_2)$ such that $(q_1,q_2,p_1,p_2)$ is a solution of (\ref{Garnier2}).
\end{prop}
\begin{proof}
Since we know no rational parametrisation of $(q_1,q_2)$ we consider the "symmetrised" system:
\begin{displaymath}
\left\lbrace \begin{array}{clc} 
\dr_{t_k} S_\mathbf{q} & = (\dr_{\mathbf{p}_1} + \dr_{\mathbf{p}_2})H_k & k=1,2 \\ 
\dr_{t_k} P_\mathbf{q} &  =   (\mathbf{q}_2\dr_{\mathbf{p}_1} + \mathbf{q}_1\dr_{\mathbf{p}_2})H_k & k=1,2 \\
\dr_{t_k} S_\mathbf{p} &  = - (\dr_{\mathbf{q}_1} + \dr_{\mathbf{q}_2})H_k& k=1,2 \\
\dr_{t_k} \gamma &  = \dfrac{-1}{(\mathbf{q}_1-\mathbf{q}_2)^2} ( (\mathbf{q}_1-\mathbf{q}_2)(\dr_{\mathbf{q}_1} + \dr_{\mathbf{q}_2})+(\mathbf{p}_1-\mathbf{p}_2)(\dr_{\mathbf{p}_1} + \dr_{\mathbf{p}_2}))H_k& k=1,2 \\
\end{array} \right. \, ,
\end{displaymath}
where $S_\mathbf{p} := \mathbf{p}_1+\mathbf{p}_2$ and $\gamma = \dfrac{\mathbf{p}_1-\mathbf{p}_2}{\mathbf{q}_1-\mathbf{q}_2}$. To obtain this we first had to consider the variable $\delta := \mathbf{q}_1-\mathbf{q}_2$ and then eliminate it using the fact that all expressions obtained had even degree in $\delta$ and that $\delta^2 = S_\mathbf{q}^2-4P_\mathbf{q}$.

Assume that $(p_1,p_2)$ are two algebraic functions such that $(q_1,q_2,p_1,p_2)$ is a solution of (\ref{Garnier2}). Using the first two equations with $k=1$ one then gets $S_\mathbf{p}$ and $\gamma$ as functions of $\dr_{t_1} S_q$ and $\dr_{t_1} P_q$ which in turn (see Lemma \ref{lem:RatPar}) are rational functions of $(\alpha,t)$, namely:
\begin{align*}
\gamma=&-\dfrac{(\lambda_0+\lambda_1-1)(\alpha+1)(\alpha-{z}^2)^2(\alpha-1)}{2\alpha(\alpha-{z})(\alpha+{z})({z}+1)({z}-1)}\, , \\
Sp  = &\dfrac{(\alpha-{z}^2)}{2\alpha(\alpha-{z})(\alpha+{z})({z}+1)({z}-1)}\hat{S_p} \: ,
\end{align*}
with 
\begin{align*}
\hat{S_p} = & (\lambda_0 + 2\lambda_1-1)\alpha^3 \\
& + ( (2\lambda_0+\lambda_-2)z^2 - (3\lambda_0+\lambda_1-3) )\alpha^2\\
& + ( (\lambda_0-3\lambda_1+1)\alpha + (\lambda_0-1) )z^2 \: .
\end{align*}
This completes the rational parametrisation of all relevant variables and allows us check that $(S_q,P_q,S_p,\gamma)$ indeed satisfies the above system.
\end{proof}

We can describe more precisely the rational surface parametrising $q_1$ and $q_2$ as follows. Using the equations linking $(S_t,P_t)$ to $(P_q,S_q)$ and Gröbner bases one show that $S_q$ is root of a degree four polynomial with coefficients depending on $S_t,P_t$ (and thus on $t_1,t_2$) and that $P_q$ can be computed as a polynomial in $S_t,P_t$ and $S_q$. Therefore, there exists a polynomial $P \in \CC[X,T_1,T_2]$ of degree four in its first variable such that $P(S_q,t_1,t_2) = 0$ and so if one sets \begin{displaymath}
\Sigma := \lbrace x,t_1,t_2 \in \PP^1 \, | \, P(x,t_1,t_2) = 0 \rbrace
\end{displaymath}
then the projection $p : \Sigma \rightarrow \PP^1_{t_1} \times \PP^1_{t_2}$ is a fourfold ramified covering, whose holonomy we can fully describe.

\begin{prop}
The holonomy representation into $\mathfrak{S}_4$ of the covering $p$ is trivial at $t_1 = t_2$ and is a double transposition at $t_i=0,1,\infty$ ($i=1,2$).
\end{prop}
\begin{proof}
Since $(S_q,P_q)$ is solution of a Garnier system, we know that this covering can only ramify over $t_i = 0, 1 , \infty$ ($i=1,2$) or $t_1 = t_2$. To better understand the way it does, let us look into its holonomy representation, which is a mapping from the fundamental group $G$  of the complement of the ramification locus in $\PP^1 \times \PP^1$ into the symmetric group $\mathfrak{S}_4$. By explicitly factorising the polynomial $P$ over all components of the possible ramification locus one gets that:
\begin{itemize}
\item over $t_i = 0$ ($i=1,2$) the polynomial has two double roots;
\item over $t_i = 1$ ($i=1,2$), the situation is the same
\item over $t_i = \infty$ ($i=1,2$), there is only one order four root;
\item over $t_1 = t_2$ the polynomial has four simple roots (the covering doesn't actually ramify there).
\end{itemize}
If one looks (for example) at the restricted polynomial $P(Sq,t_1,7)$ one can see that its discriminant has a double root at $t_1 = 1$ and that the same is true should one exchange the roles of $t_1$ and $t_2$; this means that the holonomy around $t_i = 0,1$ is a double transposition. Moreover, it takes two elementary transforms to turn the ramification at infinity into two double roots with the discriminant in $Sq$ having a double root there. The holonomy being invariant under birational morphisms, it is also a double transposition.
\end{proof}

\begin{coro}
The complex surface $\Sigma$ is rational.
\end{coro}
\begin{proof}
By setting $t_1$ or $t_2$ to any value distinct from $0,1, \infty$, one gets a fourfold covering from some curve $C$ onto $\PP^1$ ramifying over $0,1$ and $\infty$. The Riemann--Hurwitz formula yields that the curve $C$ is of genus zero, meaning that it is necessarily a rational curve. This proves that the surface $\Sigma$ is a fibration over $\PP^1$ with general fibre isomorphic to $\PP^1$ and so is in fact rational (see for example \cite{Kollar}).
\end{proof}

\section{Lotka--Volterra foliations}\label{sec:LV}

\noindent In order to prove Theorem \ref{TheoC}, let us first define the following notion (see \cite{Scardua}).

\begin{defi}[Transversally projective foliation]
Let $M$ be a smooth projective complex manifold; a codimension one foliation $\Fr$ on $M$ (defined by a Frobenius--integrable nonzero rational one--form $\omega_\Fr$) is said to be \emph{transversally projective} if there exist two rational one--forms $\alpha, \beta$ over $M$ such that
\begin{displaymath}
 \dd + \begin{pmatrix}
\alpha & \beta \\ 
\omega_\Fr & -\alpha
\end{pmatrix} 
 \end{displaymath} 
defines a flat $\mathfrak{sl}_2(\CC)$--connection over the rank two trivial bundle $\CC^2 \times M$.
\end{defi}

If one looks at the restriction $\omega$ of the Riccati one--form $\Ric(\nabla)$ to $(w=\infty)$ one obtains a codimension one transversally projective foliation $\Fr$ over the projective plane $\PP^2$; indeed, if
\begin{displaymath}
 \Ric(\nabla) = \dd w + \omega w^2 + 2 \alpha w + \beta
\end{displaymath}
then
\begin{displaymath}
 \dd + \begin{pmatrix}
\alpha & \beta \\ 
\omega & -\alpha
\end{pmatrix} 
 \end{displaymath}
 is gauge--equivalent to $\nabla$ and as such is a flat $\mathfrak{sl}_2(\CC)$--connection over $\CC^2 \times \PP^2$. The one--form $\omega$ can be written in the affine chart $\CC^2_{x,y}\subset \PP^2$ described in Subsection \ref{sec:Rk2VB} as:
\begin{displaymath}
\omega = ((2\lambda_0+\lambda_1)x+\lambda_1(y-1)) y\dd x  -((\lambda_0+2\lambda_1)y+\lambda_0(x-1)) x \dd y
\end{displaymath}
This foliation's invariant locus contains the singular locus of $\nabla$, namely the quintic $\Qr$ and has seven order one singularities, namely (in homogeneous coordinates $[x:y:t]$ chosen so that our usual affine chart corresponds to $t=1$) $[0:0:1],\,[0:1:1],\,[1:0:1], \, [\lambda_1^2 : \lambda_0^2:(\lambda_0+\lambda_1)^2], \, [1:1:0], \, [0:1:0]$ and $[1:0:0]$. Also note that this foliation only depends on the quotient $\lambda := \dfrac{\lambda_0}{\lambda_1}$; indeed it is equivalent to:
\begin{displaymath}
((2\lambda+1)x+y-1) y\dd x  -((\lambda+2)y+\lambda(x-1)) x \dd y = 0 \; .
\end{displaymath}
Also note that every singular point of the above foliation lies on the quintic $\Qr$.
\begin{center}
\begin{figure}

\begin{pspicture}(0,-3.85)(11.02,3.85)
\definecolor{color3}{rgb}{0.00392156862745098,0.00392156862745098,0.00392156862745098}
\definecolor{color14}{rgb}{1.0,0.3764705882352941,0.0}
\pscircle[linewidth=0.04,dimen=outer](6.37,-0.28){1.79}
\psline[linewidth=0.04cm,linecolor=color3](0.0,-3.83)(7.5,3.23)
\psline[linewidth=0.04cm,linecolor=color3](4.16,3.83)(10.96,-1.71)
\psline[linewidth=0.04cm,linecolor=color3](0.04,-3.47)(11.0,-1.11)
\usefont{T1}{ptm}{m}{n}
\rput(4.5628123,-1.9){$\Cr$}
\psdots[dotsize=0.12,linecolor=color14](6.3,2.11)
\psdots[dotsize=0.12,linecolor=color14](7.46,1.13)
\psdots[dotsize=0.12,linecolor=color14](5.18,1.01)
\psdots[dotsize=0.12,linecolor=color14](6.72,-1.99)
\psdots[dotsize=0.12,linecolor=color14](10.38,-1.23)
\psdots[dotsize=0.12,linecolor=color14](0.52,-3.37)
\end{pspicture} 

\caption{Singular locus for the foliation $\Fr$.}
\end{figure}
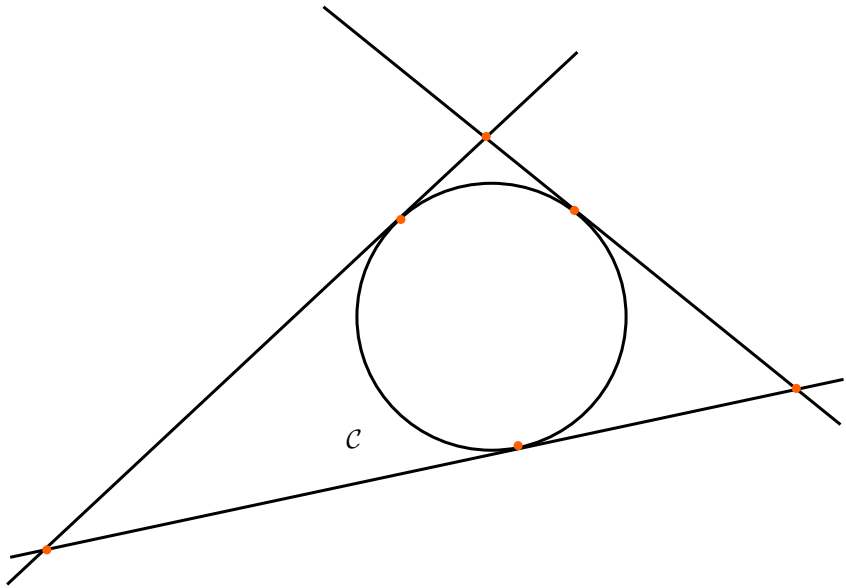
\end{center}

Now define, given three complex parameters $(A,B,C)$, the Lotka-Volterra vector field over $\CC^3$ (with coordinates $x,y,t$) as $\mathrm{LV}(A,B,C) := V_x \dr_x +V_y \dr_y + V_t \dr_t$, where:
\begin{displaymath}
V_x := x(Cy+t), \qquad V_y := y(At+x) \qquad \text{ and } \qquad Vt := t(Bx+y) \; .
\end{displaymath}
This system traditionally comes from the study of a "food chain" system with $3$ species preying on each other in a cycle. One can then \cite{MO, MO_cor} consider the foliation defined by both $\mathrm{LV}(A,B,C)$ and the radial vector field $R:= x \dr_x + y \dr_y + t \dr_t$: it is the codimension one foliation over $\CC^3$ associated with the one--form
\begin{displaymath}
\omega_0 := (yV_t-tV_y)\dd x + (tV_x - xV_t)\dd y + (xV_y - yV_x) \dd t \; .
\end{displaymath}

\subsection{Proof of Theorem \ref{TheoC}}
To prove Theorem \ref{TheoC}, one needs only show that the foliations defined by the one--forms $\omega$ and $\omega_0^\prime := {\omega_0}_{|_{t=1}}$ are the same in some affine chart. Each of the aforementioned one--forms has four singular points, namely 
\begin{displaymath}
(0,0) \; ,  \left( \dfrac{1}{B}, 0 \right), \; (0,A) \; \text{ and } \left( \dfrac{A(C-1)+1}{C(B-1)+1}, \dfrac{B(A-1)+1}{C(B-1)+1} \right) \quad \text{ for }\omega_0^\prime
\end{displaymath}
and
\begin{displaymath}
(0,0) \; ,  \left( 1, 0 \right), \; (0,1) \; \text{ and } \left( \dfrac{\lambda_1^2}{(\lambda_0+\lambda_1)^2}, \dfrac{\lambda_0^2}{(\lambda_0+\lambda_1)^2} \right) \quad \text{ for }\omega \; .
\end{displaymath}
We then submit $\omega_0^\prime$ to an affine change of coordinates to send its first three singular points onto $(0,0)$, $(1,0)$ and $(0,1)$. A necessary condition for the two forms to define the same foliation is then that their fourth singularities be equal; after computation we find that one must have:
\begin{equation}
\dfrac{B(A(C-1)+1)}{C(B-1)+1} = \dfrac{\lambda_1^2}{(\lambda_0+\lambda_1)^2}
\end{equation}
and
\begin{equation}
\dfrac{B(A-1)+1}{A(C(B-1)+1)} = \dfrac{\lambda_0^2}{(\lambda_0+\lambda_1)^2} \; .
\end{equation}
Solving the two above equations, one obtains a rational parametrisation of $A$ and $C$ by $B$, namely:
\begin{displaymath}
A = \dfrac{(B-1)\lambda_1}{(2 \lambda_0+\lambda_1)B} \quad \text{ and } \quad C =- \dfrac{2B(\lambda_0+\lambda_1)^2 + \lambda_0 \lambda_1}{\lambda_0\lambda_1(B-1)} \; .
\end{displaymath}
A necessary and sufficient condition for the two associated foliation to coincide is that $\omega \land \omega_0^\prime = 0$; using this and the above parametrisation one gets that $B$ must be equal to $-\dfrac{\lambda_0}{\lambda_0 + \lambda_1}$ and thus obtains the first par of Theorem \ref{TheoC}. 

Conversely, direct computation shows that any degree two foliation over $\PP^2$ whose invariant locus contains the quintic $\Qr$ can be written in the affine chart $(s,p)$ as
\begin{displaymath}
((\gamma_1+2\gamma_2) x + \gamma_1 (y -1)) y \dd x - ( (2\gamma_1+\gamma_2) y + \gamma_2(x-1)) x \dd y \; 
\end{displaymath}
with $\gamma_1, \gamma_2 \in  \CC$. In particular, such a foliation automatically comes from the monodromy representation of one of our connections $\nabla_{\lambda_0, \lambda_1}$, with $\lambda_0 = \gamma_2$ and $\lambda_1 = \gamma_1$.

\begin{rema}$ $
\begin{enumerate}
\item The relation $ABC = 1$ obtained in Theorem \ref{TheoC} can be seen intuitively as coming from the order $3$ symmetry of the quintic $\Qr$: indeed if one denotes by $J$ the homographic order $3$ transform defined on $\PP^1$ by
\begin{displaymath}
z \mapsto - \dfrac{1}{1+z}
\end{displaymath}
then one has $$(A,B,C) = \left( \dfrac{\lambda_1}{\lambda_0}, J\left(\dfrac{\lambda_1}{\lambda_0}\right), J^2\left(\dfrac{\lambda_1}{\lambda_0}\right)\right) \; . $$
\item The two variables Lotka--Volterra system is usually defined as being following "prey--predator" differential system:
\begin{displaymath}
\left\lbrace \begin{array}{l}
x^\prime =x ( \alpha + \beta y )\\ 
y^\prime  = y (  \gamma + \delta x )  
\end{array} \right. \; ,
\end{displaymath}
to model an ecosystem where $x$ preys on $y$. However, the plane foliation associated with this system cannot be conjugate to the one associated with $\omega$ as it has two double singular points whereas $\omega$ has seven simple singularities. Thus this gives some form of justification to the fact that we chose to consider a three variables system in this paragraph (as opposed to the more "natural" two variables one).
\end{enumerate}
\end{rema}

\subsection{Invariant curves}

The invariant locus of the family of foliations presented here does not have normal crossings, hence the Cerveau--Lins Neto bound on the degree ($\deg(\Fr) + 2$, see \cite{CerLins}) does not apply here. Furthermore one may note that (for generic parameters $\lambda_0, \lambda_1)$ the foliation $\Fr$ has simple singularities at the tangency locus of the conic $\Cr$ and the three invariant lines. Moreover, we have the following result.

\begin{prop}
The foliation $\Fr$ admits, for $\lambda_0,\lambda_1 \in \QQ$, invariant algebraic curves of arbitrarily high (depending on $\lambda_0/\lambda_1$) degree.
\end{prop}
\begin{proof}
The section $(w = \infty)\subset \PP^1 \times \PP^2$ that we used to define our foliations lifts through $\pi:\PP^1 \times \PP^1 \xrightarrow[]{2:1}\PP^2$ (see Subsection \ref{sec:Rk2VB}) to the section $(z=1)$ of the trivial bundle $X \times \PP^1$ (see Subsection \ref{sec:Rk1PB}) and so the foliation itself lifts (in our usual local chart) to:
\begin{displaymath}
(\Fr^\prime) \qquad \lambda_0 \left(\dfrac{\dd u_0}{u_0}-\dfrac{\dd u_1}{u_1} \right) + \lambda_1 \left(\dfrac{\dd u_0}{u_0-1}-\dfrac{\dd u_1}{u_1-1} \right) = 0 \; \; .
\end{displaymath}
If one looks at rational values of $\lambda_0$ and $\lambda_1$, one gets a foliation $\Fr$ with finite holonomy which as a consequence admits a rational first integral. Moreover, in that particular case every leaf is an algebraic invariance curve and it is possible to find these with arbitrarily high degree (for varying $\lambda_0, \lambda_1$). For example, if $\lambda_0 = n \geq 1$ is a positive integer and if we set $\lambda_1 = 1$ then a simple computation shows that the curve
\begin{displaymath}
(C_n) \qquad u_0^n(u_0-1) - u_1^n(u_1-1) = 0 
\end{displaymath}
on $X$ is invariant under $\Fr^\prime$. An induction then shows that this curve is the pullback by $\pi$ of a degree $2+n$ curve on $\PP^2$ and so we get an invariant curve of such degree for the foliation $\Fr$ corresponding with the parameters $(n,1)$. 
\end{proof}

\begin{rema}
Note however that this is a slightly weaker example than the ones given in \cite{ALN} as the local type of our singularities depends on the parameter $\lambda_0/\lambda_1$.
\end{rema}

\section{Proof of Theorem \ref{TheoB}} In this paragraph, we prove that our family of monodromy representations cannot be generically obtained through a pullback method \cite{Diar1, Diar2} by showing that it does not factor through a curve \cite{CorSim}.

\subsection{First case: $\lambda_0$ and $\lambda_1$ are not linearly dependant over $\ZZ$}

Suppose that we have some complex projective curve $C$, a divisor $\delta = t_1+ \ldots + t_k $ in $C$, an algebraic mapping $f :\PP^2 - f^{-1}(\delta) \rightarrow C - \delta$ and a representation $\tilde{\rho}$ of the fundamental group of $C-\delta$ into $PSL_2(\CC)$ satisfying the conditions stated in Definition \ref{def:factorCurve}. In particular, the diagram
\begin{displaymath}
\xymatrix{
\pi_1(C-\delta,x_0)  \ar[rd]^{\tilde{\rho}} & \pi_1( \PP^2 - f^{-1}(\delta)) \ar[d]^{\mathrm{P}\circ\rho\circ m} \ar[l]_{ f_*}\\
 & PSL_2(\CC) 
}
\end{displaymath}
commutes. Since the ramified covering $\pi : X \xrightarrow[]{2:1} \PP^2$ is unramified between $X-D$ and $\PP^2 - \Qr$, where $D$ is the divisor in $X$ made of the six lines $u_0,u_1=0,1,\infty$ and the diagonal $\Delta = (u_0=u_1)$, then the fundamental group $\pi_1(X-D)$ is realised as a subgroup of $\Gamma$. This means that if one sets $\phi := f \circ \pi$ one has such a diagram:
\begin{displaymath}
\xymatrix{
\pi_1(C-\delta,x_0) \ar[rd]^{\tilde{\rho^\prime}} &  \pi_1(X-\phi^{-1}(\delta)) \ar[d]^{\rho^\prime} \ar[l]_{ \phi_*}\\
 & PSL_2(\CC) 
} \; .
\end{displaymath}

Now let $L$ be a generic horizontal line in $X$ (i.e of the form $(u_1=c)$, with $c\neq 0,1,\infty$); since $f$ is algebraic the restricted map $\phi_{|_L}$ extends as a ramified covering $\phi_L : L \rightarrow C$ with topological degree equal to some $d \geq 1$. The line $L$ is isomorphic to $\PP^1$, so the Riemann--Hurwitz formula forces the genus of the curve $C$ to be equal to zero; as such we can assume without loss of generality that $\phi_L$ is a $d$--fold covering of the projective line over itself. Moreover, one has that $\phi_L^* \delta$ must contain $\lbrace 0,1,c,\infty \rbrace$.

The representation $\tilde{\rho^\prime}$ must induce infinite order monodromy about at least one loop in $C-\delta$, say $\gamma_0$, or else all elements in the image of $\rho$ would be of finite order. This means that $M:=\tilde{\rho^\prime}(\gamma_{0})$ is a infinite-order element in $PSL_2(\CC)$.

Let us assume that there are at least two distinct elements $\gamma$ and $\gamma^\prime$ in the fibre of $(\phi_L)_*$ above $\gamma_0$; then both $\rho^\prime(\gamma)$ and $\rho^\prime(\gamma^\prime)$ must be powers of $M$. This gives us a relation between words in the matrices
\begin{displaymath}
\begin{pmatrix}
a_0 & 0 \\ 
0 & a_0^{-1}
\end{pmatrix} \; , \qquad \begin{pmatrix}
a_1 & 0 \\ 
0 & a_1^{-1}
\end{pmatrix} \text{ and }\begin{pmatrix}
a_0a_1 & 0 \\ 
0 & (a_0a_1)^{-1} \end{pmatrix}\quad ,
\end{displaymath}
where $a_j = e^{-i\pi \lambda_j}$. Since generically $\lambda_0$ and $\lambda_1$ are not linearly dependant, this is impossible; hence we have that the fibre $(\phi_L)_*^{-1}(\gamma_0)$ may only contain one element. This implies that $\phi_L$ ramifies totally over (at least) three points in $C$ and so the Riemann--Hurwitz formula yields that $\phi_L$ must be one--to--one.

Let $u \in \PP^1$ and set $h_{u} \in PSL_2(\CC)$ to be the Möbius transform sending the ramification locus of $\phi_{(u_1=u)}$ onto ${0,1,\infty}$; up to composing it with $(u_0,u_1) \mapsto (h_{u_1}(u_0),u_1)$ we can assume that $\phi$ is exactly the first projection $\mathrm{pr}_1 : X \rightarrow \PP^1$. However if one looks at the restriction of $\phi$ to some vertical line then one should again generically obtain infinite local monodromy at three points, which is impossible with $\mathrm{pr}_1$, thus concluding the proof.

\subsection{Second case: there exists $(p,q)$ in $\ZZ^2\setminus \lbrace (0,0) \rbrace$ such that $p \lambda_0 + q \lambda_1 = 0$} We can assume that at least one of $\dfrac{\lambda_0}{\lambda_1}$ or $\dfrac{\lambda_1}{\lambda_0}$ is a rational number, therefore the transversally projective foliation $\Fr$ introduced in Section \ref{sec:LV} has finite monodromy and so admits some rational first integral $g : \PP^2 \rightarrow \PP^1$. Using Subsection 4.4 in \cite{LTP}, one deduces that the transversally projective structure $(\beta, \alpha, \omega)$ associated with $\Fr$ is equivalent to one of the form $(\tilde{\beta},0,\dd g)$ with the following relations (see \cite{LTP}, Subsection 4.1):
\begin{displaymath}
 \tilde{\beta} \land \dd g = 0 \quad \text{ and } \quad \dd \tilde{\beta} = 0 \;.
\end{displaymath}
The first relation implies that $\tilde{\beta}$ must be of the form $\tilde{\beta} = f \dd g$ for some rational $f : \PP^2  \rightarrow \PP^1$; using the second relation one then gets that
\begin{equation}
\dd f \land \dd g = 0 \; .
\end{equation}
Using standard results from birational geometry (see for example Theorem II.7 in \cite{Beauville}) one obtains that there exists a complex surface $M$ and a finite sequence $\mathfrak{b} : M \rightarrow \PP^2 $ of blow--ups such that $\mathfrak{g} := g \circ \mathfrak{b}$ is a holomorphic function on $M$. Moreover, if we set $\mathfrak{f} := f \circ \mathfrak{b}$ then we must have
\begin{equation}\label{eqn_ThB}
\dd \mathfrak{f} \land \dd \mathfrak{g} = 0 \;.
\end{equation}
It then follows from Stein's factorisation theorem that there exists a complex curve $C$, a ramified covering $r : C \rightarrow \PP^1$ and a fibration $\phi : M \rightarrow C$ with connected fibres such that the following diagram
\begin{displaymath}
\xymatrix{
M \ar[d]_{\mathfrak{g}} \ar[r]^\phi & C \ar[ld]^r \\
\PP^1 &
 }
\end{displaymath}
commutes. This means that locally on any sufficiently small analytic open set $U$ the covering $r$ gives an orbifold coordinate $x$ on the curve $C$ and there exists a biholomorphism $h$ between $U \times F$ and $\phi^{-1}(U)$, where $F$ is a connected complex curve, such that for all $(x,y) \in U \times F$, $g\circ h(x,y) = x$. Therefore relation (\ref{eqn_ThB}) yields:
\begin{displaymath}
\dd (\mathfrak{f}\circ h) \land \dd x = 0 \; .
\end{displaymath}
Thus $\mathfrak{f}$ depends locally only on $\mathfrak{g}$ and since the fibres of $\phi$ are connected one can conclude using analytic continuation that $\mathfrak{f}$ is globally a function of $\mathfrak{g}$. In the end, this implies that the transversally projective structure associated with $\Fr$ is equivalent to $(f(g)\dd g,0, \dd g)$ and so factors through through the algebraic map associated with $f$ on $\PP^2 - I$, where $I$ is the indeterminacy locus of $f$.

\backmatter

\nocite{*}
\bibliographystyle{smfalpha}
\bibliography{2014-garnierEx6}

\end{document}